\documentclass[english]{smfart}
\usepackage{amsmath, amssymb, amsthm, amscd,smfenum,xspace,mathrsfs,url,upref,verbatim}
\usepackage[utf8]{inputenc} 
\usepackage[T1]{fontenc}
\usepackage{dsfont}

\usepackage{pstricks}
\usepackage{graphicx}
\usepackage{ tipa }
\usepackage{mathabx}
\usepackage{stmaryrd}

\usepackage[all,cmtip]{xy}
\usepackage{xfrac} 

\usepackage{ textcomp }

\usepackage{tikz-cd}
\usepackage{tikz}

\DeclareMathAlphabet{\mathpzc}{OT1}{pzc}{m}{it}

\let\hat=\widehat
\let\tilde=\widetilde
\numberwithin{equation}{subsection}
 
\newtheorem{theorem}{Theorem}

\newtheorem{proposition}[equation]{Proposition}
\newtheorem{lemme}[equation]{Lemma}
\newtheorem{corollaire}[equation]{Corollary}

\theoremstyle{remarque}
\newtheorem{remarque}[equation]{Remark}

\theoremstyle{fact}

\newtheorem{definition}[equation]{Definition}

\DeclareMathOperator{\res}{res}

\DeclareMathOperator{\DR}{DR}

\DeclareMathOperator{\Diag}{Diag}
\DeclareMathOperator{\AS}{AS}

\DeclareMathOperator{\Hom}{Hom}

\DeclareMathOperator{\adj}{adj}
\DeclareMathOperator{\ord}{ord}

\DeclareMathOperator{\id}{id}
\DeclareMathOperator{\Ker}{Ker}

\DeclareMathOperator{\Spec}{Spec}

\def\cartesien{\ar@{}[rd]|{\square}}

\DeclareMathOperator{\RHom}{\mathcal{RH}om}

\DeclareMathOperator{\GL}{GL}

\DeclareMathOperator{\Sk}{Sk}

\DeclareMathOperator{\End}{End}

\DeclareMathOperator{\univ}{univ}

\DeclareMathOperator{\Irr}{Irr}
\DeclareMathOperator{\Sol}{Sol}

\DeclareMathOperator{\cosk}{cosk}

\DeclareMathOperator{\Jump}{Jump}
\DeclareMathOperator{\Isom}{Isom}

\DeclareMathOperator{\Char}{char}

\DeclareMathOperator{\Id}{Id}
\DeclareMathOperator{\MCHom}{\mathcal{H}om}

\DeclareMathOperator{\Real}{Re}

\DeclareMathOperator{\sk}{sk}
\DeclareMathOperator{\St}{St}

\DeclareMathOperator{\AdSk}{AdSk}

\DeclareMathOperator{\Int}{Int}
\DeclareMathOperator{\Lie}{Lie}
\DeclareMathOperator{\Sh}{Sh}
\DeclareMathOperator{\iso}{iso}

\usepackage{enumerate}

\author[J.-B.~Teyssier]{Jean-Baptiste Teyssier}
\curraddr{KU Leuven. Department of Mathematics. Celestijnenlaan 200B, Leuven. Belgium}
\email{jeanbaptiste.teyssier@kuleuven.be}
\title{Skeletons and moduli of Stokes torsors}

\usepackage[english]{babel}
\begin{document}
\begin{abstract}
We prove an analogue for Stokes torsors  of Deligne's skeleton conjecture and deduce from it the representability of the functor of relative Stokes torsors by an affine scheme of finite type over $\mathds{C}$. This provides, in characteristic 0, a local analogue of the existence of a coarse moduli for skeletons with bounded ramification, due to Deligne. 
As an application, we use the geometry of this moduli to derive quite strong finiteness results for integrable systems of differential equations in several variables which did not have any analogue in one variable.
\end{abstract}


\maketitle

Consider the following linear differential equation $(E)$ with polynomial coefficients
$$
p_n\frac{d^n f}{ dz^n}+p_{n-1} \frac{d^{n-1} f}{ dz^{n-1}}+\dots + p_1\frac{d f}{dz}+p_0 f=0
$$
If $p_n(0)\neq 0$,  Cauchy theorem asserts that a holomorphic solution to $(E)$ defined on a small disc around $0$ is equivalent to the values of its $n$ first derivatives at $0$. If $p_n(0)=0$, holomorphic solutions to $(E)$ on a small discs  around $0$ may always be zero. Nonetheless, $(E)$ may have formal power series solutions. The \textit{Main asymptotic development theorem} \cite{SVDP}, 
due to Hukuhara and Turrittin asserts that for a direction $\theta$ emanating from 0, and for  a formal power series solution $f$, there is a  sector $\mathcal{S}_{\theta}$ containing $\theta$ such that $f$ can be "lifted" in a certain sense to a holomorphic solution $f_{\theta}$ of $(E)$ on $\mathcal{S}_{\theta}$. We  say that \textit{$f_{\theta}$ is asymptotic to $f$ at $0$}. If $f_{\theta}$ is analytically continued around 0 into  a solution $\tilde{f}_{\theta}$ of $(E)$ on the sector $\mathcal{S}_{\theta^{\prime}}$, where $\theta^{\prime}\neq \theta$, it may be that the asymptotic development of $\tilde{f}_{\theta}$ at $0$ is not $f$ any more. This is the \textit{Stokes phenomenon}. 
As a general principle, the study of  $(E)$ amounts to the study of its "formal type" and the study of how asymptotic developments of solutions jump via analytic continuation around $0$. To organize these informations, it is traditional to adopt a linear algebra point of view. 
\\ \indent
The equation $(E)$  can be seen as a \textit{differential module}, i.e. a finite dimensional vector space $\mathcal{N}$ over the field $\mathds{C}\{z\}[z^{-1}]$ of convergent Laurent series, endowed with a $\mathds{C}$-linear endomorphism $\nabla : \mathcal{N}\longrightarrow \mathcal{N}$ satisfying the Leibniz rule. 
In this language, solutions of $(E)$ correspond to elements of $\Ker \nabla$ (also called \textit{flat sections of $\nabla$}). 
Furthermore, a differential equation with same "formal type" as $(E)$ corresponds to a differential module $\mathcal{M}$ with an isomorphism of \textit{formal} differential modules $\iso : \mathcal{M}_{\hat{0}}\longrightarrow \mathcal{N}_{\hat{0}}$. Since $\iso$ can be seen as a formal flat section of the differential module $\Hom(\mathcal{M}, \mathcal{N})$, the main asymptotic development theorem applies to it. The lifts of $\iso$ to sectors thus produce a cocycle $\gamma:=(\iso_\theta\iso_{\theta^{\prime}}^{-1})_{\theta,\theta^{\prime}\in S^1}$ with value into the sheaf of sectorial automorphisms  of 
$\mathcal{N}$ which are asymptotic to $\Id$ at $0$. This is the \textit{Stokes sheaf} of $\mathcal{N}$, denoted by  $\St_{\mathcal{N}}(\mathds{C})$. A fundamental result of Malgrange \cite{MaSi} and Sibuya \cite{SI} implies that  $(\mathcal{M}, \iso)$ is determined by the torsor under $\St_{\mathcal{N}}(\mathds{C})$ associated to the cocycle $\gamma$. Hence, \textit{Stokes torsors encode in an algebraic way analytic data and classifying differential equations amounts to studying Stokes torsors}. As a result, the study of Stokes torsors is meaningful.
\\ \indent
In higher dimension, the role played by differential modules is played by \textit{good meromorphic connections}. We will take such a connection $\mathcal{N}$  defined around  $0\in \mathds{C}^n$ to be of the shape \begin{equation}\label{shape}
\mathcal{E}^{a_1} \otimes \mathcal{R}_{a_1}\oplus \dots \oplus \mathcal{E}^{a_d} \otimes \mathcal{R}_{a_d}
\end{equation}
where the $a_i$ are meromorphic functions with poles contained in a normal crossing divisor $D$, where $
\mathcal{E}^{a_i}$ stands for the rank one connection $(\mathcal{O}_{\mathds{C}^n, 0}(\ast D), d-da_i)$, and where the $\mathcal{R}_{a_i}$ are regular connections. 
Note that from works of Kedlaya \cite{Kedlaya1}\cite{Kedlaya2} and  Mochizuki \cite{Mochizuki2}\cite{Mochizuki1}, every meromorphic connection is (up to ramification) formally isomorphic at each point to a connection of the form \eqref{shape}  at the cost of blowing-up enough the pole locus. If $(r_i, \theta_i)_{i=1, \dots, n}$ are the usual polar coordinates on $\mathds{C}^n$, the Stokes sheaf  $\St_{\mathcal{N}}$ of $\mathcal{N}$  is a sheaf  of complex unipotent  algebraic groups over the torus  $\mathds{T}:=(S^1)^n$ defined by $r_1=\cdots = r_n=0$. 
 \\ \indent
By a \textit{good $\mathcal{N}$-marked connection}, we mean the data $(\mathcal{M}, \nabla, \iso)$ of a  good meromorphic connection $(\mathcal{M}, \nabla)$ around  $0$ endowed with an isomorphism of formal connections at the origin $\iso: \mathcal{M}_{\hat{0}}\longrightarrow \mathcal{N}_{\hat{0}}$. As in dimension 1, Mochizuki \cite{MochStokes}\cite{Mochizuki1} showed that good $\mathcal{N}$-marked connections are determined by their associated Stokes torsor, so we consider them as elements in $H^{1}(\mathds{T}, \St_{\mathcal{N}}(\mathds{C}))$.\\ \indent
Since $\St_{\mathcal{N}}$ is a sheaf of complex algebraic groups, its sheaf of $R$-points $\St_{\mathcal{N}}(R)$  is a well-defined sheaf of groups on $\mathds{T}$ for any commutative $\mathds{C}$-algebra $R$. Consequently, one can consider the functor of relative Stokes torsors $R\longrightarrow H^1(\mathds{T}, \St_{\mathcal{N}}(R))$, denoted by $H^1(\mathds{T}, \St_{\mathcal{N}})$. Following a strategy designed by Deligne, Babbitt and Varadarajan \cite{BV} proved that $H^1(S^1, \St_{\mathcal{N}})$ is representable by an affine space. Hence in dimension 1, the set of torsors under $\St_{\mathcal{N}}$ has a structure of a complex algebraic variety. \\ \indent
The interest of this result is to \textit{provide a framework in which questions related to differential equations can be treated with the apparatus of algebraic geometry}. 
This might  look  like a wish rather than a documented fact since the local theory of linear differential equations is fully understood in dimension $1$ by means of analysis. In dimension $\geq 2$ however, new phenomena appear and  this geometric perspective seems relevant. As we will show, the representability of $H^1(\mathds{T}, \St_{\mathcal{N}})$ in any dimension implies for differential equations quite strong finiteness results which have no counterparts in dimension 1 and which seem out of reach with former technology. Thus, we prove the following
\begin{theorem}\label{bigtheorem}
The functor $H^{1}(\mathds{T}, \St_{\mathcal{N}})$ is representable by an affine scheme of finite type over $\mathds{C}$.
\end{theorem}
Before explaining how the proof relates to Deligne's skeleton conjecture, let us describe two applications to finiteness results. \\ \indent
Suppose that $\mathcal{N}$ is \textit{very good}, that is, for functions $a_i$, $a_j$ appearing in \eqref{shape} with $a_i\neq a_j$, the difference $a_i- a_j$ has poles along all the components of the divisor $D$ along which $\mathcal{N}$ is localized. Let $V$ be a manifold containing $0$ and let us denote by  $\mathcal{N}_{V}$ the restriction of the connection $\mathcal{N}$ to $V$. We prove the following 
\begin{theorem}\label{restriction}
If $V$ is transverse to every irreducible component of $D$, there is only a finite number  of equivalence classes of good $\mathcal{N}$-marked connections with given restriction to $V$. Furthermore, this number depends only on $\mathcal{N}$ and on $V$.
\end{theorem}
This theorem looks like a weak differential version of Lefschetz's theorem. A differential Lefschetz theorem would assert that for a generic choice of $V$,  good $\mathcal{N}$-marked connections are determined by their restriction to $V$.  It is a hope of the author that such a question is approachable by geometric means using 
the morphism of schemes 
\begin{equation}\label{res}
\res_V: H^{1}(\mathds{T}, \St_{\mathcal{N}})\longrightarrow H^{1}(\mathds{T}^{\prime}, \St_{\mathcal{N}_{V}})
\end{equation}
induced by the restriction to $V$.    \\ \indent
To give flesh to this intuition, let us indicate how geometry enters the proof of Theorem \ref{restriction}. Since unramified morphisms of finite type have finite fibers, it is enough to show that  $\mathcal{N}$-marked connections lie
in the unramified locus of $\res_V$, which is the locus where
 the tangent map of $\res_V$ is injective. We show in \ref{tanandirr}  a canonical identification
\begin{equation}\label{relation}
T_{(\mathcal{M}, \nabla, \iso)} H^1(\mathds{T}, \St_\mathcal{N})\simeq \mathcal{H}^1 (\Sol \End\mathcal{M})_0
\end{equation}
where the left-hand side denotes the tangent space of $H^1(\mathds{T}, \St_\mathcal{N})$ at $(\mathcal{M}, \nabla, \iso)$ and where $\mathcal{H}^1 \Sol$ denotes the first cohomology sheaf of the solution complex of a $\mathcal{D}$-module. Note that the left-hand side of \eqref{relation} is \textit{algebraic}, whereas the right-hand side is \textit{transcendental}. From that we deduce a similar interpretation to the kernel $\Ker T_{(\mathcal{M}, \nabla, \iso)} \res_V$ and prove its vanishing using a perversity theorem due to Mebkhout \cite{Mehbgro}. \\ \indent
Using an invariance theorem due to Sabbah \cite{SabRemar}, we further prove  the following rigidity result:
\begin{theorem}\label{rigidity}
Suppose that $D$ has at least two components and that $\mathcal{N}$ is very general. Then there is only a finite number of equivalence classes of good $\mathcal{N}$-marked connections.
\end{theorem}
In this statement, very general means that $\mathcal{N}$ is very good and that the monodromy of each regular constituent contributing to  $\mathcal{N}$ in \eqref{shape}  has eigenvalues away from a denombrable union of strict Zariski closed   subsets of an affine space. \\ \indent
Let us finally explain roughly the proof of Theorem  \ref{bigtheorem}. The main idea is to import and prove a conjecture from the field of Galois representations.  Let $X$ be a smooth variety over a finite field of characteristic $p>0$, and let $\ell\neq p$  be a prime number. To any $\ell$-adic local system $\mathcal{F}$
on $X$ up to semi-simplification, one can associate its 
\textit{skeleton} $\sk \mathcal{F}$, that is the collection of restrictions of $\mathcal{F}$ to curves 
drawn on $X$ endowed with compatibilities at intersection points. It has bounded ramification 
at infinity in an appropriate sense, see \cite{HenKerz}. As a consequence 
of Cebotarev's density theorem, the assignment $\mathcal{F}\longrightarrow \sk 
\mathcal{F}$ is injective and Deligne conjectured that any skeleton with 
bounded ramification comes from an $\ell$-adic local system. Building on the work of 
Wiesend \cite{Wiesend}, this has been proved in the tame case by Drinfeld 
\cite{DriDel} and for arbitrary ramification in the rank one case by Kerz and Saito 
\cite{KerzSaito}. \\ \indent
Back to characteristic $0$, let $\mathscr{C}$ be the family of smooth curves $i: C \hookrightarrow \mathds{C}^n$ containing $0$. For $C\in \mathscr{C}$, passing to polar coordinates induces    a map $(r, \theta)\longrightarrow i(r, \theta)$. Restricting to $r=0$ thus produces an embedding in $\mathds{T}$ of the circle $S^1_C$ of directions  inside $C$  emanating from $0$. 
We define  for every $\mathds{C}$-algebra $R$ a  \textit{Stokes skeleton relative to $R$} as a  collection of Stokes torsors $(\mathcal{T}_C \in  H^{1}(S^1_C, \St_{\mathcal{N}_C}(R)))_{C\in \mathscr{C}}$
endowed with compatibilities at points of $\mathds{T}$ where two circles $S^1_{C_1}$ and $S^1_{C_2}$ intersect, where $C_1, C_2\in \mathscr{C}$.
We observe that a naive version of Deligne's conjecture stating that every Stokes skeleton comes from a Stokes torsor is false, and  introduce a 
 combinatorial condition \ref{admissible} satisfied by skeletons of  Stokes torsors that we call \textit{admissibility}. We finally prove the following
\begin{theorem}\label{equcat}
Restriction induces a bijection between $H^{1}(\mathds{T}, \St_{\mathcal{N}}(R))$ and the set $\St_{\mathcal{N}}(R)\text{-}\Sk_{\mathscr{C}}$ of  $\St_{\mathcal{N}}(R)$-admissible skeletons relative to $R$.
\end{theorem}
The Stokes skeleton functor $R\longrightarrow \St_{\mathcal{N}}(R)\text{-}\Sk_{\mathscr{C}}$ is easily seen to be representable by an affine scheme by Babbitt-Varadarajan's theorem. Hence, to prove Theorem \ref{bigtheorem}, the whole point is to show that the same is true when admissibility is imposed. To do this, we  first use a general theorem \cite[6.3]{ArtinOnStack}, see also \cite[10.4]{LMB} and \cite[04S6]{SP} to obtain the representability of $H^{1}(\mathds{T}, \St_{\mathcal{N}})$ by an algebraic space of finite type over $\mathds{C}$. Finally, for a
finite family $\mathscr{C}_f\subset \mathscr{C}$ of carefully chosen curves, we show that the formation of the $\mathscr{C}_f$-skeleton
$$
H^{1}(\mathds{T}, \St_{\mathcal{N}})\longrightarrow \St_{\mathcal{N}}\text{-}\Sk_{\mathscr{C}_f}
$$
is a closed immersion.

\subsection*{Acknowledgement}

This paper comes from the observation that a particular case of a conjecture introduced in \cite{phdteyssier} reduces to the equality between two numbers that were likely to be the dimensions at two points of the tangent spaces of a moduli problem involving marked connections. This equality was obtained and generalized by C. Sabbah in \cite{SabRemar} via analytical means. I thank him for his interest and for useful remarks. \\ \indent
Representability by an algebraic space was first derived using  a general representability theorem from Artin \cite[5.4]{ArtinOnStack}. I thank M. Brion et P. Eyssidieux for pointing out that the more specific \cite[6.3]{ArtinOnStack} is enough, thus improving the readability of this paper. I also thank A. Pulita for useful comments on the introduction.\\ \indent
I thank the Mathematics department of the Hebrew University of Jerusalem and the Katholieke Universiteit of Leuven for outstanding working conditions. This work has been achieved with the support of Freie Universität/Hebrew University of Jerusalem joint post-doctoral program, ERC \text{Harmonic Analysis and $\ell$-adic sheaves}, and the long term structural funding-Methusalem grant of the Flemish Government.

\section{Generalities on Stokes torsors}

\subsection{}\label{notation}
Let $D$ be a germ of normal crossing divisor at $0\in \mathds{C}^n$ and let $i: D\longrightarrow \mathds{C}^n$ be the canonical inclusion. Let $D_1, \dots, D_m$ be the irreducible components of $D$. Let $\mathcal{I}$ be a good set of irregular values with poles contained in $D$ in the sense of \cite{Mochizuki1}. For every $a\in \mathcal{I}$, set 
$$
\mathcal{E}^{a}=(\mathcal{O}_{\mathds{C}^n, 0}(\ast D), d-da)
$$
We fix once for all a germ of unramified good meromorphic flat bundle of rank $r$ of the form
\begin{equation}\label{formalmodel}
(\mathcal{N}, \nabla_{\mathcal{N}}):=\bigoplus_{a\in \mathcal{I}}\mathcal{E}^{a} \otimes \mathcal{R}_a
\end{equation}
where $\mathcal{R}_a$ is a germ of regular meromorphic connection with poles along $D$. We also fix a basis $\mathbf{e}_a$ of $\mathcal{R}_a$.
We denote by $E_a$ the $\mathds{C}$-vector space generated by $\mathbf{e}_{a}$, $E:= \oplus_{a\in \mathcal{I}} E_a$, $i_a : E_a\longrightarrow E$ the canonical inclusion and $p_a : E \longrightarrow E_a$ the canonical projection.\\ \indent
Note that in dimension $>1$, an arbitrary meromorphic connection may not have a formal model of the shape \eqref{formalmodel}, but we know from works of Kedlaya \cite{Kedlaya1}\cite{Kedlaya2} and  Mochizuki \cite{Mochizuki2}\cite{Mochizuki1} that such a model exists (up to ramification) at each point after enough suitable blow-ups.


\subsection{Recollection on asymptotic analysis} \label{rappelasan}
As references for asymptotic ana\-lysis and good connections, let us mention \cite{Maj}, \cite{Stokes} and  \cite{Mochizuki1}. 
For $i=1, \dots, m$, let $\tilde{X}_i\longrightarrow \mathds{C}^n$ be the real blow-up of $\mathds{C}^n$ along $D_i$ and let $p:\tilde{X}\longrightarrow \mathds{C}^n$ be the fiber product of the $\tilde{X}_i$. We have  $\tilde{X}\simeq ([0, +\infty [ \times S^{1})^{m}\times \mathds{C}^{n-m}$ and $p$ reads
$$
((r_k,\theta_k )_{k},y)\longrightarrow ((r_k e^{i\theta_k} )_{k},y)
$$
In particular, we have an open immersion $j_{D}: \mathds{C}^n\setminus D\longrightarrow \tilde{X}$ and $\mathds{T}:=p^{-1}(0)$ is isomorphic to $(S^1)^m$.\\ \indent
Let $\mathcal{A}$ be the sheaf on $\mathds{T}$ of holomorphic functions on $\mathds{C}^n\setminus D$ admitting an asymptotic development along $D$. Let $\mathcal{A}^{<D}$ be the subsheaf of $\mathcal{A}$ of functions asymptotic to $0$ along $D$. These sheaves are endowed with a structure of $p^{-1}\mathcal{D}_{\mathds{C}^n,0}$-module and there is a canonical \textit{asymptotic development} morphism
$$
\xymatrix{
\AS_0:  \mathcal{A} \ar[r] & \mathds{C}\llbracket z_1, \dots, z_n\rrbracket 
}
$$
where  $\mathds{C}\llbracket z_1, \dots z_n\rrbracket$ has to be thought of as the constant sheaf on $\mathds{T}$. \\ \indent
For a germ $(\mathcal{M}, \nabla)$ of flat meromorphic connection at $0$, the module $p^{-1}\mathcal{M}$ makes sense in a neighbourhood of $\mathds{T}$ in $\tilde{X}$. Thus, 
$$
\tilde{\mathcal{M}}:=\mathcal{A}\otimes_{\mathcal{O}_{\mathds{C}^n, 0}}  p^{-1}\mathcal{M}
$$
is a $p^{-1}\mathcal{D}_{\mathds{C}, 0}$-module  on $\mathds{T}$.  Let $\DR \tilde{\mathcal{M}}$ be the De Rham complex of $ \tilde{\mathcal{M}}$ and let  $\DR^{<D} \mathcal{M}$ be the De Rham complex of $ \tilde{\mathcal{M}}$ with coefficients in $\mathcal{A}^{<D}$. \\ \indent
For $\xi \in \mathcal{H}^{0}\DR \mathcal{M}_{\hat{0}}$, we denote by $\mathcal{H}^{0}_{\xi}\DR  \tilde{\mathcal{M}}$ the subsheaf of $\mathcal{H}^{0}\DR  \tilde{\mathcal{M}}$ of sections $\tilde{\xi}$ which are asymptotic to $\xi$, that is 
$$
(\AS_0\otimes \id_{ \mathcal{M}})(\tilde{\xi})=\xi
$$ 
We set $$\Isom_{\iso}( \mathcal{M},  \mathcal{N}):=\mathcal{H}^{0}_{\iso}\DR \MCHom( \tilde{\mathcal{M}}, \tilde{\mathcal{N}})
$$
Mochizuki's asymptotic development theorem \cite{Mochizuki1}\footnote{see also \cite{HienInv} for an account of the proof.} implies that the sheaf $\Isom_{\iso}(\mathcal{M}, \mathcal{N})$ is a torsor under $\St_{\mathcal{N}}(\mathds{C})$ defined below.

\subsection{Stokes hyperplane}
For $a, b \in \mathcal{I}$, the function
$$
F_{a,b}:=  \Real(a-b)|z^{-\ord(a-b)}|
$$
induces a $C^{\infty}$ function on $\partial \tilde{X}$. We denote by $\mathcal{H}_{a,b}$ its vanishing locus on $\mathds{T}$. The 
$\mathcal{H}_{a,b}$ are the \textit{Stokes hyperplanes of $\mathcal{I}$}. Concretely, $$a-b=f_{ab}z^{\ord(a-b)}$$ with $f_{ab}(0)\neq 0$ and $\ord(a-b)=-(\alpha_{ab}(1), \dots, \alpha_{ab}(m))$, where $\alpha_{ab}(k)\geq 0$ for $k=1, \dots, m$. Hence,  $F_{a,b}$ induces
\begin{equation}\label{Fabinduit}
((r_k, \theta_k)_{1\leq k \leq m}, (z_k)_{m+1\leq k \leq n})\longrightarrow \Real f_{ab}(r_k e^{i \theta_k}, z_k) e^{-i\sum_{k=1}^{m}\alpha_{ab}(k) \theta_k}
\end{equation}
Set $f_{ab}(0)=r_{ab}e^{i\theta_{ab}}$ with $0\leq \theta_{ab}<2\pi$. The restriction of \eqref{Fabinduit} to $\mathds{T}$ is 
\begin{equation}\label{FabinduitsurT} 
(\theta_1, \dots, \theta_m)\longrightarrow \cos(\theta_{ab}-\sum_{k=1}^{m}\alpha_{ab}(k) \theta_k)
\end{equation} 
From now on, we see $\mathds{T}$ as $\mathds{R}^m/2\pi \mathds{Z}^m$ and we denote by $\pi : \mathds{R}^m \longrightarrow \mathds{T}$ the canonical projection.
For $l\in \mathds{Z}$, let $\mathcal{H}_{a,b}(l)$ be the hyperplane of $\mathds{R}^{m}$ given by  
$$\sum_{k=1}^{m}\alpha_{ab}(k) \theta_k= \theta_{ab}+\frac{\pi}{2}+ l\pi
$$
and define
$$
Z_{ab}=\displaystyle{\bigsqcup_{l\in \mathds{Z}} \mathcal{H}_{a,b}(l)}
$$
Then $\mathcal{H}_{a,b}=\pi(Z_{ab})$.\\ \indent
If $\mathcal{S}$ is a product of strict intervals, $\mathcal{S}$ is homeomorphic via $\mathds{R}^{m}\longrightarrow \mathds{T}$ to an open $U(\mathcal{S})\subset \mathds{R}^m$. Since $Z_{ab}=Z_{ab}+ 2\pi \mathds{Z}^m$ we have
$$
\mathcal{H}_{a,b}\cap \mathcal{S} =  \pi(Z_{ab})\cap \mathcal{S}=\pi(Z_{ab}\cap U(\mathcal{S}))\simeq Z_{ab}\cap U(\mathcal{S})=\displaystyle{\bigsqcup_{l\in \mathds{Z}} \mathcal{H}_{a,b}(l)\cap U(\mathcal{S})}
$$
Hence, the  connected components of $ \mathcal{H}_{a,b}\cap \mathcal{S}$ correspond via $U(\mathcal{S})\simeq \mathcal{S}$ to the non empty intersections between $U(\mathcal{S})$ and the $\mathcal{H}_{a,b}(l)$.

\subsection{The Stokes sheaf. Definition and local structure}\label{localstructure}
For $a, b \in \mathcal{I}$ and $\mathcal{S}$ any subset of $\mathds{T}$, we define following \cite{Mochizuki1} a partial order $<_\mathcal{S}$ on $\mathcal{I}$ as follows
\begin{equation}\label{order}
a<_\mathcal{S} b \text{ if and only if } F_{a,b}(x)<0 \text{ for all } x\in \mathcal{S}
\end{equation}
We use notations from \ref{notation}. For every $R\in \mathds{C}$-alg, we define $\St_{\mathcal{N}}(R)$ as the subsheaf of   $R\otimes_{\mathds{C}} (j_{D\ast}\mathcal{H}^{0}\DR\End \mathcal{N})_{|\mathds{T}}$ of sections of the form $\Id+f$ where $p_ a f i_b = 0$ unless  $a<_\mathcal{S} b$.  \\ \indent
Suppose that $\mathcal{S}$ is contained in a product of strict open intervals. For $a\in \mathcal{I}$, the regular connection $\mathcal{R}_a$ admits in the basis $\mathbf{e}_a$ a fundamental matrix of flat sections $F_{a}$ on $\mathcal{S} $. Set $F:=\oplus_{a\in \mathcal{I}}F_a$ and $D:= \oplus_{a\in \mathcal{I}} a \Id$. For every $h\in (j_{D\ast}\mathcal{H}^{0}\DR\End \mathcal{N} )_{|\mathcal{S}}$, a standard computation shows that the derivatives of 
$e^{-D} F^{-1} h  F e^{D}$ are $0$. We thus have a well-defined isomorphism
\begin{equation}\label{inj}
\xymatrix{
(j_{D\ast}\mathcal{H}^{0}\DR\End \mathcal{N} )_{|\mathcal{S}}\ar[r]^-{\sim} & \underline{\End E}
}
\end{equation}
For every $\mathds{C}$-vector space $I$, we obtain  a  commutative diagram
\begin{equation}\label{tri}
\xymatrix{
\Gamma(\mathcal{S}, I\otimes_{\mathds{C}}j_{D\ast}\mathcal{H}^{0}\DR\End \mathcal{N})\ar[r]^-{\sim} & I\otimes_{\mathds{C}}\End E \\
I\otimes_{\mathds{C}}\Gamma(\mathcal{S}, j_{D\ast}\mathcal{H}^{0}\DR\End \mathcal{N})     \ar[ru]_-{\sim} \ar[u]
 }
\end{equation}
Hence, the canonical morphism 
\begin{equation}\label{sortleR}
I\otimes_{\mathds{C}}\Gamma(\mathcal{S}, j_{D\ast}\mathcal{H}^{0}\DR\End \mathcal{N})   \longrightarrow\Gamma(\mathcal{S}, I\otimes_{\mathds{C}}j_{D\ast}\mathcal{H}^{0}\DR\End \mathcal{N})
\end{equation}
is an isomorphism. Applying this to $I=R$, we see that \eqref{tri} identifies $\Gamma(\mathcal{S},\St_{\mathcal{N}}(R))$ with the space  of  $h\in  R\otimes_{\mathds{C}}\End E$ such that $p_ a (h-\Id) i_b = 0$ unless $a<_\mathcal{S} b$.  In particular, $\St_{\mathcal{N}}(R)$ is a sheaf of unipotent algebraic groups over $R$. \\ \indent 
If $\mathcal{S}^{\prime}\subset \mathcal{S}$ are as above and if $R\longrightarrow S$ is a morphism of rings, the following diagram 
$$
\xymatrix{
\Gamma(\mathcal{S} , \St_{\mathcal{N}}(R)) \ar[rr] \ar[ddd] \ar[dr]   & &  \Gamma(\mathcal{S}^{\prime}, \St_{\mathcal{N}}(R)) \ar[ddd] \ar[dl]   \\
 &    R \otimes_{\mathds{C}}   \End E    \ar[d]           &      \\
 &    S \otimes_{\mathds{C}}   \End E         & \\
 \Gamma(\mathcal{S} , \St_{\mathcal{N}}(S)) \ar[rr] \ar[ur]   & & \Gamma(\mathcal{S}^{\prime}, \St_{\mathcal{N}}(S)) \ar[ul]   
}
$$
commutes. Hence, horizontal arrows are injective. If $R\longrightarrow S$  is injective, vertical arrows are injective and we have further 
\begin{equation}\label{uneegalite}
\Gamma(\mathcal{S} , \St_{\mathcal{N}}(R))  =\Gamma(\mathcal{S}^{\prime}, \St_{\mathcal{N}}(R))  \cap  \Gamma(\mathcal{S} , \St_{\mathcal{N}}(S))  
\end{equation}
\subsection{Restriction to curves}
Let $\iota : C\longrightarrow \mathds{C}^n$ be a germ of smooth curve passing through $0$ and non included in $D$. Let $\tilde{\iota}:\tilde{C}\longrightarrow \tilde{X}$ be the induced morphism at the level of the real blow-ups and. We still denote  by $\tilde{\iota}:\partial\tilde{C}\longrightarrow \mathds{T}$ the induced morphism at the level of the boundaries. Note that $\tilde{\iota}$ is injective. We set $\iota^{\ast}\mathcal{I}:=\{a\circ \iota, a\in \mathcal{I}\}$. By goodness property of $\mathcal{I}$, restriction to $C$ induces a preserving order bijection $\mathcal{I}\longrightarrow \iota^{\ast}\mathcal{I}$, that is for every $x\in \partial\tilde{C}$ and $a, b\in \mathcal{I}$, we have $a<_{\tilde{\iota}(x)} b$ if and only if $a\circ \iota <_{x} b \circ \iota$. Hence, the canonical isomorphism 
$$
\xymatrix{
\tilde{\iota}^{-1}(j_{D\ast}\mathcal{H}^{0}\DR\End \mathcal{N}) \ar[r]^-{\sim}    &    (j_{0\ast}\mathcal{H}^{0}\DR\End \mathcal{N}_{C}    )_{|\partial \tilde{C}}    
}
$$
deduced from Cauchy-Kowaleska theorem for flat connections induces  for every $R\in \mathds{C}$-alg a canonical isomorphism
$$
\xymatrix{
\tilde{\iota}^{-1}\St_{\mathcal{N}}(R) \ar[r]^-{\sim}    &    \St_{\mathcal{N}_C}(R)
}
$$
compatible with \eqref{inj}.

\subsection{Preferred matricial representations in dimension 1}\label{prefered}
Let us now restrict to the dimension 1 case and let $d$ and $d^{\prime}$ be two consecutive Stokes line of $\mathcal{I}$. Let $a_1, \dots, a_k$ be the elements of $\mathcal{I}$ noted in increasing order for the total order $<_{]d, d^{\prime}[}$. In the basis $\mathbf{e}:=(\mathbf{e}_{a_1}, \dots, \mathbf{e}_{a_k})$, the morphism \eqref{inj} identifies $g\in \Gamma(]d, d^{\prime}[ , \St_{\mathcal{N}}(R))$  with the subgroup of $\GL_r(R)$ of upper-triangular matrices with only $1$ on the diagonal. Let $I$ be a strict open interval meeting $]d, d^{\prime}[$. For $i\in \llbracket 1, r \rrbracket$, let $j_i \in \llbracket 1, k \rrbracket$ such that $\mathbf{e}_i$ is an element of $\mathbf{e}_{a_{j_i}}$. Note that $j_i$ increases with $i$. We denote by $\Jump_{\mathcal{N}}(I)$ the set of $( i_1,i_2)$,  $1 \leq i_1<i_2\leq r$,  such that $j_{i_1}<j_{i_2}$ and $a_{j_{i_1}} \nless_I a_{j_{i_2}}$.


\subsection{Stokes torsors}
For $\mathcal{T}  \in H^{1}(\mathds{T}, \St_{\mathcal{N}}(R))$ and a map of ring $\varphi: R\longrightarrow S$, we denote by $\mathcal{T}(S)$ the push-forward of $\mathcal{T}$ to $S$. Concretely, if $\mathcal{T}$  is given by a cocycle $(g_{ij})$, a cocycle for $\mathcal{T}(S)$ is  $(\varphi(g_{ij}))$. There is a canonical morphism of sheaves $\mathcal{T}\longrightarrow \mathcal{T}(S)$ equivariant for the morphism of sheaves of groups $\St_{\mathcal{N}}(R)\longrightarrow \St_{\mathcal{N}}(S)$. If $t$ is a section of $\mathcal{T}$, we denote by $t(S)$ the associated section of $\mathcal{T}(S)$. \\ \indent
For $\mathcal{T}  \in H^{1}(\mathds{T}, \St_{\mathcal{N}}(R))$, let
$T_{\mathcal{T}}H^{1}(\mathds{T}, \St_{\mathcal{N}})$ 
be the tangent space of $H^{1}(\mathds{T}, \St_{\mathcal{N}})$ at $\mathcal{T}$. By definition, this is the set of  $\mathcal{T}^{\prime}  \in H^{1}(\mathds{T}, \St_{\mathcal{N}}(R[\epsilon]))$ such that 
$\mathcal{T}^{\prime}(R)=\mathcal{T}$.

\subsection{Automorphisms of Stokes torsors}\label{auto}
In this subsection, we give a proof of the following
\begin{proposition}\label{autoId}
Stokes torsors have no non trivial isomorphisms.
\end{proposition}
\begin{proof}
Using restriction to curve, we are left to treat the one dimensional case. To simplify notations, we denote by $H^1$ the affine space representing the functor $H^{1}(S^{1}, \St_{\mathcal{N}})$ and  by $\mathds{C}[H^1]$ its algebra of functions. Let $d_1, \dots , d_N$ be the Stokes lines of $\mathcal{N}$ indexed consecutively by $\mathds{Z}/N\mathds{Z}$. We denote by 
$$\St_{\mathcal{N}, d_i}^{H^1}:=H^1\times_{\mathds{C}} \St_{\mathcal{N}, d_i}$$ 
the base change of the complex algebraic group $\St_{\mathcal{N}, d_i}$ to $H^1$. Let $\mathcal{T}^{\univ}_{\mathcal{N}}$ be the universal Stokes torsor given by Babbitt-Varadarajan representability theorem in dimension 1. For each $i\in \mathds{Z}/N\mathds{Z}$, let us choose a trivialisation of $\mathcal{T}^{\univ}_{\mathcal{N}}$ on the interval $]d_{i-1}, d_{i+1}[$ and let 
$$g_{ii+1}^{\univ}\in \Gamma\left(]d_i , d_{i+1}[, \St_{\mathcal{N}}(\mathds{C}[H^1])\right)$$ 
be the associated cocycle. Let $G\longrightarrow H^1$ be the subgroup scheme  of 
$$
\prod_{i\in \mathds{Z}/N\mathds{Z}} \St_{\mathcal{N}, d_i}^{H^1}
$$
of $N$-uples $(h_1, \dots, h_N)$ satisfying 
$$
h_i g_{ii+1}^{\univ}=g_{ii+1}^{\univ} h_{i+1}
$$
in $H^1\times_{\mathds{C}}\GL_r$. For $R\in \mathds{C}$-alg, $\mathcal{T}\in H^{1}(S^1, \St_{\mathcal{N}}(R))$ corresponds to a unique morphism of $\mathds{C}$-algebras $f: \mathds{C}[H^1]\longrightarrow  R$, and a cocycle for $\mathcal{T}$ is given by applying $f$ to the $g_{ii+1}^{\univ}$. Hence,  automorphisms of $\mathcal{T}$ are in bijection with $R$-points of $\Spec R \times_{H^1} G$. To prove \ref{autoId}, we are thus left to prove that $G$ is the trivial group scheme over $H^1$, that is, that the structural morphism of $G$ is an isomorphism.\\ \indent
As a complex algebraic group, $G$ is smooth. The affine scheme $H^1$ is smooth as well. So to prove that
$G\longrightarrow H^1$ is an isomorphism, it is enough to prove that it induces a bijection at the level of the underlying topological spaces.  This can be checked above each point of $H^1$, that is to say after base change to a field $K$ of finite type over $\mathds{C}$.  It is enough to show the bijectivity after base change to an algebraic closure $\overline{K}$ of $K$. Hence, we are left to prove Theorem \ref{autoId} for $R=\overline{K}$.  By Lefschetz principle, we can suppose $R=\mathds{C}$.  \\ \indent
Let $\mathcal{T}\in H^{1}(S^1, \St_{\mathcal{N}}(\mathds{C}))$ and let $(g_{ij})$ be a cocycle of $\mathcal{T}$ with respect to an open cover $(U_i)$. As already seen, an automorphism of  $\mathcal{T}$ is equivalent to the data of sections $h_i\in \Gamma (U_i, \St_{\mathcal{N}}(\mathds{C}))$ satisfying 
\begin{equation}\label{petiterelation}
h_i g_{ij}=g_{ij} h_j  
\end{equation} 
The cocycle $g$ defines an element of $H^{1}(S^1, \Id+ M_r( \mathcal{A}^{<0}))$. At the cost of refining the cover, Malgrange-Sibuya theorem \cite[I 4.2.1]{BV} asserts the existence of sections $x_i\in \Gamma(U_i, \GL_r(\mathcal{A}))$ such that $x_i x_j^{-1}=g_{ij}$ on $U_{ij}$. Since
$$
x_i^{-1} h_i x_i=x_j^{-1}(g_{ij}^{-1}h_i g_{ij}) x_j =  x_j^{-1} h_j x_j
$$
the $x_i^{-1} h_i x_i$ glue into a global section of $\Id+ M_r( \mathcal{A}^{<0})$. Since $\mathcal{A}^{<0}$ has no non zero global section, we deduce $h_i =\Id$ for every $i$ and \ref{autoId} is proved.
\end{proof}

\subsection{$\mathcal{I}$-good open sets}\label{Igoodopen}
A $\mathcal{I}$-good open set at  $x\in\mathds{T}$ 
is a product $\mathcal{S}$ of strict intervals containing $x$ and such that 
$$
 \left\{
    \begin{array}{ll}
     \mathcal{S}\cap \mathcal{H}_{a,b}=\emptyset &   \text{ if } x\notin \mathcal{H}_{a,b}\\
 \mathcal{S}\cap \mathcal{H}_{a,b} \text { is connected } & \text{ if }  x\in \mathcal{H}_{a,b}
    \end{array}
\right.
$$
Every $x\in\mathds{T}$ admits a fundamental system of neighbourhoods which are $\mathcal{I}$-good open sets.

\begin{lemme}\label{crittri}
For every $R\in \mathds{C}$-alg and every  $\St_{\mathcal{N}}(R)$-torsors $\mathcal{T}$, the restriction of $\mathcal{T}$ to a $\mathcal{I}$-good open set $\mathcal{S}$ at $x\in \mathds{T}$ is trivial.
\end{lemme}
\begin{proof}
Let $\mathcal{H}_{a_1, b_1}, \dots, \mathcal{H}_{a_k, b_k}$ be the Stokes hyperplanes passing through $x$. For $I\subset \llbracket 1, k \rrbracket$ non empty, we set 
$$
\mathcal{H}_{I}:=\bigcap_{i\in I}(\mathcal{H}_{a_i, b_i}\cap \mathcal{S})
$$ 
Finally, we set $\mathcal{H}_{\emptyset}=\mathcal{S}$.  Let us choose $t\in \mathcal{T}_x$. We argue by decreasing recursion on $d$ that $t$ extends to 
$$
\mathcal{H}(d):= \bigcup_{I\subset \llbracket 1, k \rrbracket, |I|=d}\mathcal{H}_{I}
$$
We know by $\mathcal{I}$-goodness that $\mathcal{H}(k)$ is homeomorphic to a $\mathds{R}$-vector space. Since the order \eqref{order} is constant on $\mathcal{H}(k)$, the sheaf of group $\St_{\mathcal{N}}(R)$ is constant on $\mathcal{H}(k)$, the section $t$ extends uniquely to  $\mathcal{H}(k)$. We now suppose $d<k$ and assume that $t$ extends to $\mathcal{H}(d+1)$ into a section that we still denote by $t$. If we set 
$$
 \mathcal{H}_{I^+} :=\displaystyle{\bigcup_{i \in \llbracket 1, k \rrbracket \setminus I}} \mathcal{H}_{I\cup \{i\}} \text{ and   }    \mathcal{H}_{I}^- :=  \mathcal{H}_{I}\setminus  \mathcal{H}_{I^+}
$$
We have set theoretically
$$
\mathcal{H}(d)=\mathcal{H}(d+1)\bigsqcup \bigsqcup_{\underset{|I|=d}{I\subset \llbracket 1, k \rrbracket }}\mathcal{H}_{I}^-
$$
Hence, we have to extend $t$ to each $\mathcal{H}_{I}^-$ in a compatible way. Let $I\subset \llbracket 1, k \rrbracket$ of cardinal $d$. By admissibility, the data of  the $\mathcal{H}_{I\cup \{i\}}$, $i\in \llbracket 1, k \rrbracket\setminus I$ inside $\mathcal{H}_{I}$ is topologically equivalent to that of a finite number of hyperplanes in a  $\mathds{R}$-vector space. Hence, a connected component  
$F_I\in \pi_0(\mathcal{H}_{I}^-)$ is contractible and
$\mathcal{H}_{I^+}\subset \mathcal{H}(d+1)$
admits an open neighbourhood $U$ whose trace on $F_I$ is connected. Since the order \eqref{order} is constant on $F_I$, the  sheaf of group $\St_{\mathcal{N}}(R)$ is constant on $F_I$. Hence, $t_{|U\cap F_I}$ extends uniquely to a section $t_{F_I}$ on $F_I$. If $F_I^{\prime}\in \pi_0(\mathcal{H}_{I}^-)$ is distinct from $F_I$, 
$$
\overline{F_I}\cap \overline{F_{I}^{\prime}}\subset  \mathcal{H}_{I^+}
$$
hence $t_{| \mathcal{H}_{I^+}} $ and the   $t_{F_I}$ glue into a section $t_I$ of $\mathcal{T}$ on $\mathcal{H}_I$. For 
$I^{\prime}\subset \llbracket 1, k \rrbracket$ distinct from $I^{\prime}$ and of cardinal $d$, and $F_{I^{\prime}}\in \pi_0(\mathcal{H}_{I^{\prime}}^-)$, 
we have 
$$
\overline{F_I}\cap \overline{F_{I^{\prime}}}\subset \mathcal{H}_{I}\cap \mathcal{H}_{I^{\prime}}\subset \mathcal{H}(d+1)
$$
Hence $t$ and the $t_{I}$ and glue into a section of $\mathcal{T}$ over $\mathcal{H}(d)$ and \ref{crittri} is proved.

\end{proof}

\section{Skeletons and Stokes torsors}\label{partskeleton}
Let $X$ be a smooth real manifold. Let $\mathscr{C}$ be a set of closed curves in $X$, let $\Sh_X$ be the category of sheaves on $X$, set 
$$\Int \mathscr{C}:=\displaystyle{\bigcup_{C, C^{\prime}\in \mathscr{C}}} C \cap C^{\prime} $$ 
and for every $x\in \Int \mathscr{C}$, set
$\mathscr{C}(x):=\{C\in \mathscr{C} \text{ with } x\in C\}$.
\subsection{Definition}\label{defskeleton}
 We define the category of $\mathscr{C}$-skeleton $\Sk_{X}(\mathscr{C})$ on $X$ as the category whose objects are systems $(\mathcal{F}_\mathscr{C}, \iota_\mathscr{C})$ where $\mathcal{F}_\mathscr{C}=(\mathcal{F}_C)_{C\in \mathscr{C}}$
is a family of sheaves on the curves $C\in \mathscr{C}$, and where $\iota_\mathscr{C}$ is a collection of identifications
$$
\xymatrix{\iota_{C, C^{\prime}}(x): \mathcal{F}_{C, x}\ar[r]^-{\sim}  &  \mathcal{F}_{C^{\prime}, x}}
$$ 
for $x\in \Int \mathscr{C}$ and $C, C^{\prime}\in \mathscr{C}(x)$, satisfying
\begin{align*}
\iota_{C, C}(x)& =\id \\
  \iota_{ C, C^{\prime}}(x)             & =   \iota_{ C^{\prime}, C}^{-1}(x)  \\
   \iota_{ C, C^{\prime \prime}}(x)  &  =   \iota_{ C^{\prime}, C^{\prime \prime}}  \iota_{ C, C^{\prime }}(x) 
\end{align*}
A morphism $(\mathcal{F}_\mathscr{C}, \iota_\mathscr{C})\longrightarrow (\mathcal{G}_\mathscr{C}, \kappa_\mathscr{C})$ in $\Sk_{X}(\mathscr{C})$ is a collection of morphisms of sheaves $f_C :\mathcal{F}_C\longrightarrow \mathcal{G}_C$ such that the following diagrams commute for every $x\in \Int \mathscr{C}$ and $C, C^{\prime}\in \mathscr{C}(x)$
$$
\xymatrix{
\mathcal{F}_{C, x}\ar[r]^-{\sim}   \ar[d]_-{f_{C,x}} & \mathcal{F}_{C^{\prime}, x}  \ar[d]^-{f_{C^{\prime},x}}           \\
\mathcal{G}_{C, x}\ar[r]^-{\sim}  &   \mathcal{G}_{C^{\prime}, x}
}
$$
Restriction induces a functor $\sk_\mathscr{C} : \Sh_X \longrightarrow \Sk_{X}(\mathscr{C})$ called the \textit{$\mathscr{C}$-skeleton functor}.

\subsection{Coskeleton}
We now suppose that $\mathscr{C}$ covers $X$, that is every $x\in X$ belongs to at least one $C\in \mathscr{C}$. Take $(\mathcal{F}_\mathscr{C}, \iota_\mathscr{C})\in \Sk_{X}(\mathscr{C})$. The set
$$
E(\mathcal{F}_\mathscr{C}):=\displaystyle{\bigsqcup_{C\in \mathscr{C}, x\in C}}\mathcal{F}_{C,x}
$$
is endowed with the equivalence relation
$$
(C, x, s\in \mathcal{F}_{C,x})\sim (C^{\prime}, x, s^{\prime}\in \mathcal{F}_{C^{\prime},x})  \text{ if and only if }  s^{\prime}=\iota_{C, C^{\prime}}(x)(s)
$$
Let $E(\mathcal{F}_\mathscr{C}, \iota_{\mathscr{C}})$ be the quotient of $E(\mathcal{F}_\mathscr{C})$ by this relation.
The surjection $E(\mathcal{F}_\mathscr{C})\longrightarrow X$ induces a surjection $p: E(\mathcal{F}_\mathscr{C}, \iota_{\mathscr{C}}) \longrightarrow X$. Let
$\cosk_{\mathscr{C}}(\mathcal{F}_\mathscr{C}, \iota_{\mathscr{C}})$ be the functor associating to every open $U$ in $X$ the set of sections $s$ of $p$ over $U$ such that for every $C\in \mathscr{C}$, there exists $s_C \in \Gamma(U\cap C, \mathcal{F}_C)$ satisfying for every $x\in U\cap C$
$$
s(x)=(C,x, s_C(x))  \text{ in } E(\mathcal{F}_\mathscr{C}, \iota_{\mathscr{C}})
$$
Equivalently, if $s(C,x)$ denotes the unique representative of 
$s(x)$ associated to $(C,x)$, the above equation means $s(C,x)=s_C(x)$ in $\mathcal{F}_{C,x}$.\\ \indent
The functor $\cosk_{\mathscr{C}}(\mathcal{F}_\mathscr{C}, \iota_{\mathscr{C}})$ is trivially a presheaf  on $X$ and one checks that it is sheaf. We thus have a well-defined functor $\cosk_{\mathscr{C}} : \Sk_{X}(\mathscr{C}) \longrightarrow \Sh_X$ called the \textit{$\mathscr{C}$-coskeleton functor}.
\begin{lemme}\label{adjoint}
The functor $\cosk_{\mathscr{C}}$ is right adjoint to $\sk_{\mathscr{C}}$.
\end{lemme}
\begin{proof}
Let $\mathcal{F}\in \Sh_X$ and let  $(\mathcal{G}_\mathscr{C}, \iota_{\mathscr{C}})\in \Sk_{X}(\mathscr{C})$. 
We have to define a natural bijection
$$
\xymatrix{
\Hom_{\Sk_{X}(\mathscr{C}) }(\sk_{\mathscr{C}}\mathcal{F}, (\mathcal{G}_\mathscr{C}, \iota_{\mathscr{C}}))\ar[r]    & 
\Hom_{\Sh_{X}}(\mathcal{F}, \cosk_{\mathscr{C}}(\mathcal{G}_\mathscr{C}, \iota_{\mathscr{C}}))
}
$$
A morphism $f:= (f_C :\mathcal{F}_{|C}\longrightarrow \mathcal{G}_C)_{C\in \mathscr{C}}$ in the left-hand side induces a well-defined map
$$
\xymatrix{
E(\sk_{\mathscr{C}}\mathcal{F}) \ar[r] & E(\mathcal{G}_\mathscr{C}, \iota_{\mathscr{C}})
}
$$
associating to the class of  $(C, x, s\in \mathcal{F}_x)$ the class of  $(C, x, f_{C,x}(s)\in \mathcal{G}_{C,x})$.
Let $U$ be an open in $X$. A section $s\in \Gamma(U, \mathcal{F})$ induces a section to $E(\sk_{\mathscr{C}}\mathcal{F}) \longrightarrow X$ over $U$, from which we deduce a section $\adj(f)(s)$ of $E(\mathcal{G}_\mathscr{C}, \iota_{\mathscr{C}})\longrightarrow X$ over $U$. For $C\in \mathscr{C}$, the section $\adj(f)(s)$ is induced on $U\cap C$ by $f_C(s_{|U\cap C})\in \Gamma(U\cap C, \mathcal{G}_C)$. Hence, $\adj(f)(s)\in \Gamma(U,  \cosk_{\mathscr{C}}(\mathcal{G}_\mathscr{C}, \iota_{\mathscr{C}}))$. We have thus constructed a morphism of sheaves
$$
\adj(f) : \mathcal{F}\longrightarrow \cosk_{\mathscr{C}}(\mathcal{G}_\mathscr{C}, \iota_{\mathscr{C}})
$$
such that the following diagram commutes
$$
\xymatrix{
\mathcal{F}_{|C}   \ar[rr]^-{\adj(f)_{|C}} \ar[rrd]_-{f_C}  & &  \cosk_{\mathscr{C}}(\mathcal{G}_\mathscr{C}, \iota_{\mathscr{C}})_{|C}    \ar[d]     \\
        &         &   \mathcal{G}_C
}
$$
where the vertical morphism sends a germ of section at $x\in C$ to its unique representative in $\mathcal{G}_{C, x}$. Thus, $\adj : \Hom_{\Sk_{X}(\mathscr{C}) }(\sk_{\mathscr{C}}\mathcal{F}, (\mathcal{G}_\mathscr{C}, \iota_{\mathscr{C}}))\longrightarrow
\Hom_{\Sh_{X}}(\mathcal{F}, \cosk_{\mathscr{C}}(\mathcal{G}_\mathscr{C}, \iota_{\mathscr{C}}))$ is well-defined and injective.
It is a routine check to see that  $\adj$ is surjective.
\end{proof}

\subsection{Torsor and skeleton}\label{torsorandske}
Let $\mathcal{G}$ be a sheaf of groups on $X$. The canonical morphism 
\begin{equation}\label{caninjectG}
 \mathcal{G}\longrightarrow \cosk_{\mathscr{C}} \sk_{\mathscr{C}} \mathcal{G}
\end{equation} 
  is injective and we suppose from now on that it is also surjective.
\begin{definition}
A $\mathcal{G}$-skeleton torsor is the data of an object $(\mathcal{F}_\mathscr{C}, \iota_\mathscr{C})\in \Sk_{X}(\mathscr{C})$ such that for every $C\in \mathscr{C}$, the sheaf $\mathcal{F}_C$ is a $\mathcal{G}_{|C}$-torsor such that for every $C, C^{\prime}\in \mathscr{C}$ and for every $x\in C \cap C^{\prime}$, the following diagram commutes
\begin{equation}\label{compatcarré}
\xymatrixcolsep{6pc}\xymatrix{
 \mathcal{G}_{|C, x}\times \mathcal{F}_{C, x}\ar[d] \ar[r]^-{\iota_{C, C^{\prime}, \mathcal{G}}(x) \times\iota_{C, C^{\prime}}(x)}  &  \mathcal{G}_{|C^{\prime}, x}\times \mathcal{F}_{C^{\prime}, x} \ar[d]  \\
   \mathcal{F}_{C, x}  \ar[r]_-{\iota_{C, C^{\prime}}(x)}      &      \mathcal{F}_{C^{\prime}, x} 
}
\end{equation}
where $\iota_{C, C^{\prime}, \mathcal{G}}(x)$ is the composite of the canonical morphisms  $\mathcal{G}_{|C, x}\overset{\sim}{\longrightarrow} \mathcal{G}_{ x}\overset{\sim}{\longleftarrow}\mathcal{G}_{|C^{\prime}, x}$. We denote by $\mathcal{G}$-$\Sk_{\mathscr{C}}$ the category of $\mathcal{G}$-skeleton torsors on $X$ with respect to $\mathscr{C}$. 
\end{definition}
Let $(\mathcal{F}_\mathscr{C}, \iota_\mathscr{C})$ be a $\mathcal{G}$-skeleton torsors. The morphism \eqref{caninjectG} and the compatibilities \eqref{compatcarré} show that  $\cosk_\mathscr{C}(\mathcal{F}_\mathscr{C}, \iota_\mathscr{C})$ is endowed 
with an action of $\mathcal{G}$. Let $U$ be an open of $X$ and let $s, t\in \Gamma(U, \cosk_\mathscr{C}(\mathcal{F}_\mathscr{C}, \iota_\mathscr{C}))$. For every $C\in \mathscr{C}$ meeting $U$, the sections $s$ and $t$ are induced on $C$ by $s_C, t_C\in \Gamma(U\cap C, \mathcal{F}_C)$. Since $\mathcal{F}_C$ is a $\mathcal{G}_{|C}$-torsor, there exists a unique $g_C\in \Gamma(U\cap C,\mathcal{G})$ such that $t_C= g_C s_C$. From \eqref{compatcarré}, we see that the $(g_C)_{C\in \mathscr{C}}$ define a section of $\cosk_\mathscr{C}\mathcal{G}$ over $U$. Since \eqref{caninjectG} is supposed to be an isomorphism, we deduce $t=gs$ for a unique $g\in \Gamma(U,\mathcal{G})$. Hence,  $\cosk_\mathscr{C}(\mathcal{F}_\mathscr{C}, \iota_\mathscr{C})$  is a \textit{pseudo-torsor}. It may not be a torsor in general.

\subsection{Stokes Torsors and skeletons}\label{STandSk}

For $\theta=(\theta_{1}, \dots, \theta_{m})\in [0, 2\pi [^m$ and $\nu=(\nu_1, \dots, \nu_m)\in (\mathds{N^{\ast}})^{m}$, we define $C_{\theta, \nu}$ as the curve  of $\mathds{C}^n$ defined by 
$$
t\longrightarrow (e^{i\theta_1} t^{\nu_1}, \dots, e^{i\theta_m} t^{\nu_m}, 0)
$$
It gives rise to a curve $\partial \tilde{C}_{\theta, \nu}\simeq S^1$ on $\mathds{T}$ explicitly given by 
$$
\theta\longrightarrow (\theta_{1}+ \nu_1 \theta, \dots ,\theta_{m}+ \nu_m \theta)
$$
From this point on, we apply the previous formalism to
$$
 \left\{
    \begin{array}{ll}
        X=\mathds{T}              \\   
         \mathscr{C}=\{\partial \tilde{C}_{\theta, \nu},\theta\in [0, 2\pi [^m, \nu\in (\mathds{N^{\ast}})^{m} \}              \\
       \mathcal{G}=\St_{\mathcal{N}}(R)
    \end{array}
\right.
$$
where $R\in \mathds{C}$-alg. 
\begin{remarque}
Proposition \ref{autoId} shows that the category $\St_{\mathcal{N}}(R)\text{-}\Sk_{\mathscr{C}}$ is a setoïd, that is a groupoïd whose objects have exactly one automorphism. We still denote by $\St_{\mathcal{N}}(R)\text{-}\Sk_{\mathscr{C}}$  the set of isomorphism classes of objects in this category.
\end{remarque}
As explained in \ref{torsorandske}, the coskeleton of a $\St_{\mathcal{N}}(R)$-skeleton torsor is a $\St_{\mathcal{N}}(R)$-pseudo torsor due to the following 
\begin{lemme}
The canonical morphism  
\begin{equation}\label{canmorphism}
\St_{\mathcal{N}}(R)\longrightarrow \cosk_{\mathscr{C}} \sk_{\mathscr{C}} \St_{\mathcal{N}}(R)
\end{equation}
is an isomorphism.
\end{lemme}
\begin{proof}
It is enough to show surjectivity over an $\mathcal{I}$-good open set $\mathcal{S}$ for $x\in \mathds{T}$. As explained in \ref{localstructure}, a choice of fundamental matrix for the $\mathcal{R}_a, a\in \mathcal{I}$ on $\mathcal{S}$ induces a commutative diagram
$$
\xymatrix{
   \Gamma(\mathcal{S},\St_{\mathcal{N}}(R) ) \ar[d] \ar[r] &  \bigsqcup_{x\in \mathcal{S}} \St_{\mathcal{N}}(R)_x  \ar[d]&  \Gamma(\mathcal{S},\cosk_\mathscr{C} \sk_\mathscr{C}\St_{\mathcal{N}}(R) )  \ar[l] \ar[ld] \\     
   \GL_r(R)  \ar[r] &  \GL_r(R)^{\mathcal{S}}& 
}
$$
with injective arrows and where the bottom arrow associates to $g\in \GL_r(R)$  the function constant  to $g$ on $\mathds{T}$. Hence, we have to show that the function $\mathcal{S}\longrightarrow \GL_r(R)$ induced by a section of $\cosk_\mathscr{C} \sk_\mathscr{C}\St_{\mathcal{N}}(R)$ over $\mathcal{S}$ is constant. Since it is constant over every admissible curve and since two points in $\mathcal{S}$ can always been connected by intervals lying on admissible curves, we are done.
\end{proof}
Let $\mathcal{S}$ be an $\mathcal{I}$-good open set of $\mathds{T}$. 
A $\mathscr{C}$-polygon of $\mathcal{S}$ is the image by the canonical surjection  $\pi:\mathds{R}^{m}\longrightarrow \mathds{T}$ of a convex polygon $P\subset \mathds{R}^{m}$ with at least 3 edges  such that $\pi(P)\subset \mathcal{S}$    and the image of an edge $E$ lies  on a curve $C(E)\in \mathscr{C}$. We introduce the following
\begin{definition}\label{admissible}
A $\St_{\mathcal{N}}(R)$-skeleton torsor $(\mathcal{T_{\mathscr{C}},\iota_{\mathscr{C}} })\in  \St_{\mathcal{N}}(R)\text{-}\Sk_{\mathscr{C}}$ is said to be \textit{admissible} if for every $\mathcal{I}$-good open set $\mathcal{S}$, every $\mathscr{C}$-polygon 
$\pi : P\longrightarrow \mathcal{S}$ with edge $E_i=[x_i, x_{i+1}]$,  $i\in \mathds{Z}/N\mathds{Z}$,  every $t_i\in \Gamma(\pi(E_i), \mathcal{T}_{C(E_i)})$, $i=1, \dots, N-1$ such that
$$
t_i(\pi(x_{i+1}))=t_{i+1}(\pi(x_{i+1})) \text{  in } E(\mathcal{T}_\mathscr{C}, \iota_{\mathscr{C}})
$$
there exists a (necessarily unique) $t_N\in  \Gamma(\pi(E_N), \mathcal{T}_{C(E_N)})$ satisfying
$$
t_N(\pi(x_N))=t_{N-1}(\pi(x_N)) \text{ and }  t_N(\pi(x_1))=t_{1}(\pi(x_1))  \text{  in } E(\mathcal{T}_\mathscr{C}, \iota_{\mathscr{C}})
$$
\end{definition}
Again, the category $\St_{\mathcal{N}}(R)\text{-}\AdSk_{\mathscr{C}}$ is a setoïd. We still denote by $\St_{\mathcal{N}}(R)\text{-}\AdSk_{\mathscr{C}}$  the set of isomorphism classes of objects in this category.
\subsection{Proof of Theorem \ref{equcat}}
If we prove that the skeleton of a Stokes-torsor is admissible, and that the coskeleton of an admissible Stokes-skeleton is a torsor, we are done since the adjunction maps provided by \ref{adjoint} will automatically be isomorphisms. \\ \indent
Let $\mathcal{T}\in H^{1}(\mathds{T}, \St_{\mathcal{N}}(R))$. Let $\mathcal{S}$ be an $\mathcal{I}$-good open set, let $\pi : P\longrightarrow \mathcal{S}$ be a $\mathscr{C}$-polygon, and let $t$ be a section of $\mathcal{T}$ on $\pi(P\setminus ]x_N, x_1[)$. From \ref{crittri}, $\mathcal{T}$  admits a section $s$ on $\mathcal{S}$. Hence, there exists a section $g$ of $ \St_{\mathcal{N}}(R)$ over  $\pi(P\setminus ]x_N, x_1[)$ such that $t=g s$ and we are left to prove that $g$ extends to $\pi(P)$. For $a,b\in \mathcal{I}$ with $ \mathcal{H}_{a,b}$ meeting $\mathcal{S}$, we have to show that 
$$
a<_{\pi(P\setminus ]x_N, x_1[)}b   \Longrightarrow   a <_{\pi(]x_N, x_1[)} b
$$
By $\mathcal{I}$-goodness, $\mathcal{S}\setminus (\mathcal{S}\cap \mathcal{H}_{a,b})$ has only two connected components $C^{\pm}_{ab}$. They are convex. If $\pi(P\setminus ]x_N, x_1[) \subset C^{-}_{ab}$, so does the segment $\pi( [x_N, x_1])$ by convexity and we are done.\\ \indent
Let $(\mathcal{T_{\mathscr{C}},\iota_{\mathscr{C}} })\in  \St_{\mathcal{N}}(R)\text{-}\AdSk_{\mathscr{C}}$ and let us prove that $\cosk_{\mathscr{C}}(\mathcal{T_{\mathscr{C}},\iota_{\mathscr{C}} })$ is a torsor under $\St_{\mathcal{N}}(R)$.
For $x\in \mathds{T}$, we have to prove that the germ of $\cosk_{\mathscr{C}}(\mathcal{T_{\mathscr{C}},\iota_{\mathscr{C}} })$ at $x$ is not empty. Let 
$$
\mathcal{S}(x, \epsilon):=\prod_{i=1}^{m} ]x_i-\epsilon, x_i + \epsilon [
$$ 
be an $\mathcal{I}$-good open set at $x$, choose $s\in E(\mathcal{T}_{\mathscr{C}})$ above $x$.
Set 
$$
\mathcal{S}(x, \epsilon)^- :=\prod_{i=1}^{m} ]x_i-\epsilon, x_i  [\text{ and }  \mathcal{S}(x, \epsilon)^+ :=\prod_{i=1}^{m} ]x_i, x_i + \epsilon [
$$
We are going to "transport" the germ $s$ in two steps to the open $\mathcal{S}(x, \epsilon/2)$. We first transport it to $\mathcal{S}(x, \epsilon)^+$ and $\mathcal{S}(x, \epsilon)^-$.\\ \indent
For $C\in \mathscr{C}(x)$, let $C(x)$ be the connected component of $C\cap \mathcal{S}(x, \epsilon)$ containing $x$. By admissibility, the set $\mathcal{H}_{a,b}\cap C(x)$ is either empty or reduced to $\{x\}$. Hence, the restriction  of  $\mathcal{T}_C$ to $C(x)$ contains at most one Stokes line, which is $\{x\}$ if there is one. Lemma \ref{crittri} thus shows that the unique representative of $s$ in $\mathcal{T}_{C, x}$ extends uniquely to a section $s_{C(x)}\in \Gamma(C(x), \mathcal{T}_C)$. Two distinct $C(x)$ and $C^{\prime}(x)$ meet only at $x$, so we have a well-defined section of $E(\mathcal{T}_{\mathscr{C}}, \iota_{\mathscr{C}})\longrightarrow \mathds{T}$ above
$$
\mathcal{S}(x, \epsilon)^+ \times \{x\}\times  \mathcal{S}(x, \epsilon)^-
$$
that we still denote by $s$, noting $s(y)$ its value at a point $y$. 
We now extend $s$ to $\mathcal{S}(x, \epsilon/2)$ using the admissibility condition. Let $y \in \mathcal{S}(x, \epsilon/2)$. If $\mathbf{1}$ denotes the $m$-uple $(1, \dots, 1)$, we can choose $y_{\pm}\in \partial \tilde{C}_{y,\mathbf{1}}\cap \mathcal{S}(x, \epsilon)^{\pm}$. 
Admissibility applied to the triangle $y_-x y_+$ and the sections $s_{[y_- , x]}, s_{[x , y_+]}$ gives a unique $t_{ y_-, y_+}\in \Gamma([y_- , y_+], \mathcal{T}_{\partial\tilde{C}_{y,\mathbf{1}}})$ such that 
$$
t_{ y_-, y_+}(y_-)=s(y_-) \text{ and } t_{ y_-, y_+}(y_+)=s(y_+) \text{ in } E(\mathcal{T}_\mathscr{C}, \iota_{\mathscr{C}})
$$
If $y_+^{\prime} \in  \partial \tilde{C}_{y,\mathbf{1}}\cap \mathcal{S}(x, \epsilon)^{+}$ is another choice, the sections $t_{ y_-, y_+}$ and $t_{ y_-, y_+^{\prime}}$
coincide at $y_-$. Hence they coincide on $[y_{-}, y_+] \cap [y_{-}, y_+^{\prime}]$. Arguing similarly with $y_-$, we deduce that $t_{ y_-, y_+}$ does not depend on the choice of $y_{+}$ and $y_-$. We have thus constructed a section 
$$
t_y\in \Gamma(\partial \tilde{C}_{y,\mathbf{1}}\cap \mathcal{S}(x, \epsilon), \mathcal{T}_{\partial\tilde{C}_{y,\mathbf{1}}})
$$
We now show that $y\longrightarrow t_y(y)$ defines a section of
$\cosk_{\mathscr{C}}(\mathcal{T_{\mathscr{C}},\iota_{\mathscr{C}} })$ over $\mathcal{S}(x, \epsilon/2)$.
Let $C\in \mathscr{C}$ and let $y_0\in C\cap \mathcal{S}(x, \epsilon/2)$. It is enough to show that $y\longrightarrow t_y(y)$ is induced by a section of $\mathcal{T}_C$  on a small enough interval of $C$ contained in $\mathcal{S}(x, \epsilon/2)$ and containing $y_0$. One can choose such a non trivial interval $[y_1, y_2]$ in a way that it admits a translate $[y_{1+}, y_{2+}]$ contained in $\mathcal{S}(x, \epsilon/2)^+$ with $y_i^+\in \partial\tilde{C}_{y_i,\mathbf{1}}$ for $i=1, 2$. The situation can be depicted as follows:

\begin{figure}[h]
\begin{center}
\includegraphics[height=2.5in,width=3.0in,angle=0]{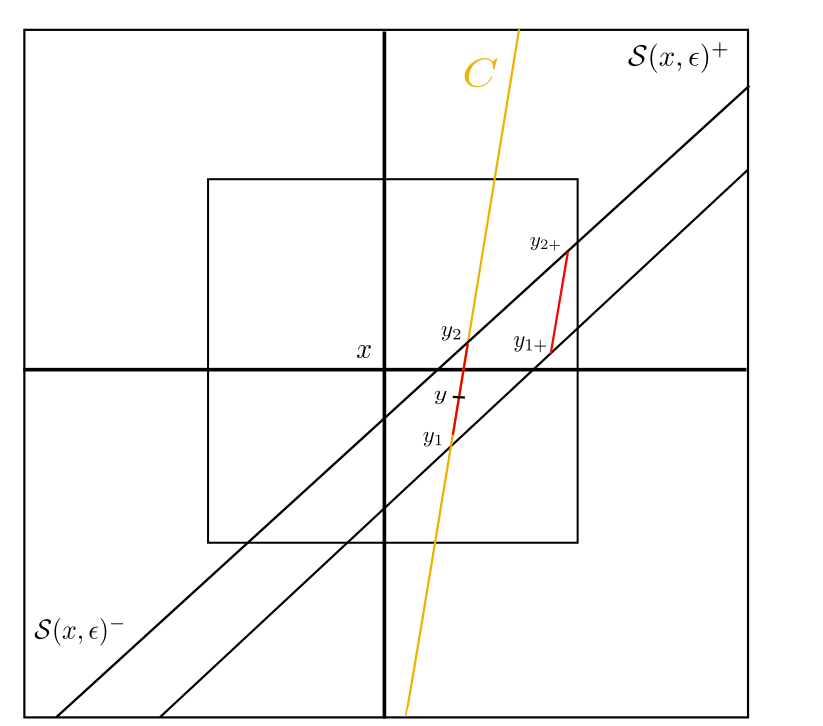}
\end{center}
\end{figure}

In particular, $[y_{1+}, y_{2+}]$ is an  interval of a translate 
$C^+$ of $C$. Since  $\mathscr{C}$ is stable by 
translation, we have $C^+ \in \mathscr{C}$. Admissibility 
applied to the triangle $y_{1+}x y_{2+}$ and the sections 
$s_{[y_{1+}, x]}, s_{[y_{2+}, x]}$ gives a unique $s_{y_{1+}, 
y_{2+}}\in \Gamma([y_{1+}, y_{2+}], \mathcal{T}_{C^+})$ such 
that 
$$
s_{y_{1+}, 
y_{2+}}(y_{1+})=s(y_{1+}) \text{ and } s_{y_{2+}, 
y_{2+}}(y_{2+})=s(y_{2+})  \text{ in } E(\mathcal{T}_\mathscr{C}, \iota_{\mathscr{C}})
$$
Note that $s_{y_{1+}, 
y_{2+}}(y)=s(y)$ for every $y\in [y_{1+}, 
y_{2+}]$. Indeed, we know by admissibility applied to $yxy_{2+}$ that there exists a unique section $s_{y,y_{2+}}\in \Gamma([y_1, y_{2+}], \mathcal{T}_{C^+})$ such that 
$$
s_{y, 
y_{2+}}(y)=s(y)  \text{ and } s_{y, 
y_{2+}}(y_{2+})=s(y_{2+})  \text{ in } E(\mathcal{T}_\mathscr{C}, \iota_{\mathscr{C}})
$$
Since $s_{y_{1+}, 
y_{2+}}$ and $s_{y, y_{2+}}$ coincide at $y_{2+}$, they are equal on $[y,y_{2+}]$, so 
$$
s_{y_{1+}, 
y_{2+}}(y)=s_{y, 
y_{2+}}(y)=s(y)     \text{ in } E(\mathcal{T}_\mathscr{C}, \iota_{\mathscr{C}})
$$
By admissibility applied to the parallelogram $y_1 y_{1+}y_{2+}y_2$ and the sections $t_{y_1|[y_1,  y_{1+}]}$, $s_{y_{1+}, y_{2+}}$ and $t_{y_2|[y_2,  y_{2+}]}$, there exists a unique section $t\in \Gamma([y_1, y_2], \mathcal{T}_C)$ such that 
$$
t(y_1)=t_{y_1}(y_1) \text{ and } t(y_2)=t_{y_2}(y_2)    \text{ in } E(\mathcal{T}_\mathscr{C}, \iota_{\mathscr{C}})
$$
We are left to show that  $t$ coincides with $y\longrightarrow t_y(y)$ on $[y_1, y_2]$. Let $y\in [y_1, y_2]$. The line $\partial\tilde{C}_{y,\mathbf{1}}$ and the segment $[y_{1+}, y_{2+}]$ meet at a point $y_+$. Admissibility applied to the parallelogram $yy_2y_{2+}y_+$ and the sections $t_{y|[y, y_+]}, t_{y_2|[y_2, y_{2+}]}, s_{y_{1+}, y_{2+}|[y_{+}, y_{2+}]}$ gives a section $t^{\prime}\in \Gamma([y, y_2], \mathcal{T}_C)$ such that 
$$
t^{\prime}(y)=t_{y}(y) \text{ and } t^{\prime}(y_2)=t_{y_2}(y_2)    \text{ in } E(\mathcal{T}_\mathscr{C}, \iota_{\mathscr{C}})
$$
Since $t^{\prime}$ and $t$ are two sections of $\mathcal{T}_{C}$ coinciding at $y_{2}$, they are equal on $[y, y_2]$. Hence
$$
t(y)=t^{\prime}(y)=t_{y}(y)  \text{ in } E(\mathcal{T}_\mathscr{C}, \iota_{\mathscr{C}})
$$
From the construction of $y\longrightarrow t_y(y)$ out of $s$, we deduce that 
\begin{equation}\label{is}
\xymatrix{
\Gamma(\mathcal{S}(x, \epsilon/2), \cosk_{\mathscr{C}}(\mathcal{T_{\mathscr{C}},\iota_{\mathscr{C}} }))\ar[r] & E(\mathcal{T_{\mathscr{C}},\iota_{\mathscr{C}} })_x
}
\end{equation}
is surjective. In particular, the left-hand side of \eqref{is} is not empty. Since the sheaf $\cosk_{\mathscr{C}}(\mathcal{T_{\mathscr{C}},\iota_{\mathscr{C}} })$ is a pseudo-torsor, \eqref{is} is also injective. Taking the colimit over $\epsilon$ gives an identification 
$$
\xymatrix{
\cosk_{\mathscr{C}}(\mathcal{T_{\mathscr{C}},\iota_{\mathscr{C}} })_x\ar[r]^-{\sim} & E(\mathcal{T_{\mathscr{C}},\iota_{\mathscr{C}} })_x
}
$$
and Theorem \ref{equcat} is proved. \\ \indent

As a corollary of Theorem \ref{equcat}, we see that admissibility is stable under base change. This is not clear a priori if one considers only a  subset of $\mathscr{C}$. Hence, the assignment $R\longrightarrow  \St_{\mathcal{N}}(R)\text{-}\AdSk_{\mathscr{C}}$ is a well-defined functor.
As a direct consequence of \ref{crittri}, we have the following
\begin{corollaire}\label{trivialsurI}
Let $(\mathcal{T}_\mathscr{C}, \iota_{\mathscr{C}})\in \St_{\mathcal{N}}(R)\text{-}\AdSk_{\mathscr{C}}$. For every $C\in \mathscr{C}$ and every interval $I\subset C$ contained in an $\mathcal{I}$-good open set, the torsor $\mathcal{T}_C$ is trivial on $I$.
\end{corollaire}

\subsection{A monomorphism into an affine scheme}\label{mono}
One can find a finite family $\mathscr{C}(\mathcal{I})\subset \mathscr{C}$ for which there exists a finite family $\mathscr{P}$ of  closed parallelotops covering $\mathds{T}$ such that for every $\mathcal{P}\in \mathscr{P}$, the following holds
\begin{enumerate}
\item $\mathcal{P}$ is contained in an open which is $\mathcal{I}$-good for  its center $x_\mathcal{P}$.
\item every edge of $\mathcal{P}$ is contained in a curve of $\mathscr{C}(\mathcal{I})$. 
\item for $\mathcal{P}_1, \mathcal{P}_2 \in  \mathscr{P}$, there exists  $x_{12}\in \mathcal{P}_1\cap  \mathcal{P}_2$, $x_1 \in \mathcal{P}_1$ and $x_2\in \mathcal{P}_2$ such that $[x_{\mathcal{P}_i}, x_i]$ and $[x_i, x_{12}]$, $i=1, 2$ lie on curves of $\mathscr{C}(\mathcal{I})$.
\end{enumerate}
Condition $(3)$ will be used in section \ref{prooftheorem} only.
\begin{lemme}\label{famillespecial}
For every $R\in \mathds{C}$-alg, the map
\begin{equation}\label{injfamillespeciale}
\xymatrix{
H^{1}(\mathds{T}, \St_{\mathcal{N}}(R)) \ar[r] & \St_{\mathcal{N}}(R)\text{-}\Sk_{\mathscr{C}(\mathcal{I})}
}
\end{equation}
is injective.
\end{lemme}
\begin{proof}
We set 
$$
\mathcal{C}(\mathcal{I}):=\bigcup_{C\in \mathscr{C}(\mathcal{I})} C
$$
Let $\mathcal{T}_1, \mathcal{T}_2\in H^{1}(\mathds{T}, \St_{\mathcal{N}}(R))$ such that $\sk_{\mathscr{C}(\mathcal{I})}(\mathcal{T}_1)=\sk_{\mathscr{C}(\mathcal{I})}(\mathcal{T}_2)$ and let 
$$
f:\mathcal{T}_{1| \mathcal{C}(\mathcal{I})}\longrightarrow \mathcal{T}_{2|\mathcal{C}(\mathcal{I})}
$$ 
be the induced isomorphism. Let $\mathcal{P} \in \mathscr{P}$. From \ref{crittri}, we can choose sections $t_i \in \Gamma(\mathcal{P},\mathcal{T}_i )$. Let $g\in \Gamma(\mathcal{P}\cap \mathcal{C}(\mathcal{I}),\St_{\mathcal{N}}(R))$ such that $f(t_1)= g t_2$ on $\mathcal{P}\cap \mathcal{C}(\mathcal{I})$. If a Stokes hyperplan $\mathcal{H}_{ab}$ meets $\mathcal{P}$, then it meets an edge of $\mathcal{P}$, so it meets $\mathcal{P}\cap \mathscr{C}(\mathcal{I})$. Hence, $g$  extends uniquely to $\mathcal{P}$, so $f$ extends uniquely to an isomorphism $f_{\mathcal{P}}$ over $\mathcal{P}$. 
Let $\mathcal{P}_1, \mathcal{P}_2\in \mathscr{P}$ with $\mathcal{P}_1 \cap \mathcal{P}_2\neq \emptyset$. The transition functions of $\St_{\mathcal{N}}(R)$ between connected sets are injective. Hence, $f_{\mathcal{P}_1}$ and $f_{\mathcal{P}_2}$ coincide on the convex $\mathcal{P}_1 \cap \mathcal{P}_2$ if they coincide at a point of $ \mathcal{P}_1 \cap \mathcal{P}_2$.  Since $\mathcal{P}_1 \cap \mathcal{P}_2$ contains a point lying on the edge of $\mathcal{P}_1$  or $\mathcal{P}_2$, we are done. Hence, the $f_{\mathcal{P}}$ glue into a global isomorphism between $\mathcal{T}_{1}$ et $\mathcal{T}_{2}$.
\end{proof}
To justify the title of this subsection, we are left to prove the representability of $\St_{\mathcal{N}}\text{-}\Sk_{\mathscr{C}(\mathcal{I})}$ by an affine scheme.
\begin{lemme}\label{finitereprske}
For every finite set of curves $\mathscr{C}_f\subset \mathscr{C}$, the functor $\St_{\mathcal{N}}$-$\Sk_{\mathscr{C}_f}$ is representable by an affine scheme of finite type over $\mathds{C}$.
\end{lemme}
\begin{proof} 
For $C\in \mathscr{C}_f$, let $\mathcal{T}^{\univ}_{\mathcal{N}_C}$ be the universal Stokes torsor for $\mathcal{N}_C$ given by Babbitt-Varadarajan representability theorem in dimension 1. For  $x\in \Int \mathscr{C}_f$ and $C\in \mathscr{C}_f(x)$, choose $t_{x, C}\in \mathcal{T}^{\univ}_{\mathcal{N}_C,x}$.  Let $R\in \mathds{C}$-alg. \\ \indent
For every $C\in \mathscr{C}_f$, let $\mathcal{T}_C$ be a $\St_{\mathcal{N}_C}(R)$-torsor and let $(\iota_{C, C^{\prime}}(x))$ be a system of compatibilities as in \eqref{compatcarré}. There exists an isomorphism  $\iso_{\mathcal{T}_C}:\mathcal{T}^{\univ}_{\mathcal{N}_C}(R)\longrightarrow \mathcal{T}_C$ and we know from Theorem \ref{autoId} that it is unique.
For $x\in \Int \mathscr{C}_f$ and  $C, C^{\prime}\in \mathscr{C}_f(x)$, define $g_{x}(\mathcal{T}_C, \mathcal{T}_{C^{\prime}})$ as the unique element of $\St_{\mathcal{N}, x}(R)$ satisfying
$$
\iota_{C, C^{\prime}}(x)(\iso_{\mathcal{T}_C}(t_{x, C}(R)))= g_{x}(\mathcal{T}_C, \mathcal{T}_{C^{\prime}}) \iso_{\mathcal{T}_{C^{\prime}}}(t_{x, C^{\prime}}(R))
$$
From Theorem \ref{autoId}, we have a well-defined injective morphism of functors
$$
\St_{\mathcal{N}}\text{-}\Sk_{\mathscr{C}_f}\longrightarrow \displaystyle{\prod_{C\in  \mathscr{C}_f} H^{1}(S^{1}, \St_{\mathcal{N}_{C}})}\times \displaystyle{\prod_{x\in \Int \mathscr{C}_f} \St_{\mathcal{N}, x}^{\mathscr{C}_f(x)^{2}}}
$$
identifying $\St_{\mathcal{N}}$-$\Sk_{\mathscr{C}_f}$ with the subfunctor of
\begin{equation}\label{grosfoncteur}
\displaystyle{\prod_{C\in  \mathscr{C}_f} H^{1}(S^{1}, \St_{\mathcal{N}_{C}})}\times \displaystyle{\prod_{x\in \Int \mathscr{C}_f} \St_{\mathcal{N}, x}^{\mathscr{C}_f(x)^{2}}}
\end{equation}
of families $(\mathcal{T}, g)$ satisfying for every $x \in \Int \mathscr{C}_f$ and every $C, C^{\prime}\in \mathscr{C}_f(x)$, 
\begin{align*}
g_{x, C, C}& =\id \\
  g_{x, C, C^{\prime}}             & =   g_{x, C^{\prime}, C}^{-1} \\
   g_{x, C, C^{\prime \prime}} &  =   g_{x, C, C^{\prime}}  g_{x, C^{\prime}, C^{\prime \prime}} 
\end{align*}
Those conditions are algebraic, so they identify $\St_{\mathcal{N}}$-$\Sk_{\mathscr{C}_f}$ with a closed subscheme of \eqref{grosfoncteur}. 
\end{proof}
Taking the limit over finite subsets of $\mathscr{C}$ gives the following
\begin{corollaire}\label{Stskrepresentable}
The functor $\St_{\mathcal{N}}$-$\Sk_{\mathscr{C}}$ is representable by an affine scheme. 
\end{corollaire}

\section{Sheaf property and tangent spaces of $H^{1}(\mathds{T}, \St_{\mathcal{N}})$}\label{Artincond}

\subsection{A technical lemma}
In this subsection, we work in dimension 1. Let $f: R\longrightarrow S$ be a morphism of rings. Let $I=]a,b[$ be a strict interval of $S^1$, let $a< d_1 < \cdots <  d_k < b$  be the Stokes lines of $\mathcal{N}$ contained in $I$. Set $d_0=a$, $d_{k+1}=b$ and $I_i:=]d_{i-1}, d_{i+1}[$ for $i=1,\dots, k$. The following lemma is obvious when $f$ is surjective. This is however the injective case that will be relevant to us.
\begin{lemme}\label{deRdansS}
For every $\mathcal{T}\in H^{1}(I, \St_{\mathcal{N}}(R))$ such that $\mathcal{T}(S)$ is trivial, there exists $t_i \in \Gamma(I_i, \mathcal{T})$ for every $i=1, \dots, k$ such that the $t_i(S)$ glue into a global section of $\mathcal{T}(S)$ on $I$.
\end{lemme}
\begin{proof}
We argue by recursion on the number of Stokes lines in $I$. The case $k=1$ is implied by \ref{crittri}. Suppose $k>1$ and let  $t_i^{\prime} \in \Gamma(I_i, \mathcal{T})$ for $i=1, \dots, k-1$ as given by the recursion hypothesis applied to $J := ]a,d_k[$. From \ref{crittri}, we can choose $t_k^{\prime}\in  \Gamma(I_k, \mathcal{T})$ and we want to modify the $t_i^{\prime}$, $i=1, \dots, k$ so that the conclusion of \ref{deRdansS} holds.  \\ \indent
Since $\mathcal{T}(S)$ is trivial, we choose $t\in \Gamma(I, \mathcal{T}(S))$ and denote by $t^{\prime}\in \Gamma(J, \mathcal{T}(S))$ the gluing of the $t_i^{\prime}(S)$ for $ i=0,\dots, k-1$. We write $t^{\prime}=g t$ with $g\in  \Gamma(J, \St_{\mathcal{N}}(S))$ and $t^{\prime}_k(S)=h t$ with $h\in  \Gamma(I_{k}, \St_{\mathcal{N}}(S))$. In the matricial representation \ref{prefered} induced by $<_{]d_{k-1}, d_k[}$, the automorphisms $g$ and $h$ correspond to upper triangular matrices. We argue by recursion on $d$ that at the cost of modifying the $t^{\prime}_i$, we can always suppose that $g$ and $h$ have the same $j$-diagonal for $j=1, \dots , d$. For $d=1$, there is nothing to do.  Suppose $d>1$ and write $g=\Id + M$ and $h=\Id + N$ where $M$ and $N$ are nilpotent matrices. On $I_{k-1}\cap I_k$ we have  $t_{k-1}^{\prime}= \gamma t_{k}^{\prime}$ with $\gamma \in  \Gamma( I_{k-1}\cap I_k , \St_{\mathcal{N}}(R))$. Hence,  on $I_{k-1}\cap I_k$
$$
gt=t^{\prime}=\gamma(S)t_k^{\prime}(S)=\gamma(S)h t
$$ 
so we deduce 
$$
gh^{-1} = \id+\displaystyle{\sum_{j=0}^{r}}(-1)^{j}  (M-N)N^j =\gamma(S)
$$
We denote by $\Diag_i(A)$ the $i$-upper diagonal of a matrix $A$.
Since $\Diag_i(M)=\Diag_i(N)$  for $i=1, \dots, d$,
we have $D_{d+1}((M-N) N^j)=0$  unless $j=0$. Hence
$$
 \left\{
    \begin{array}{ll}
        M_{i, i+d}-N_{i, i+d}=f(\gamma_{i, i+d})& \mbox{if } (i, i+d)\notin  \Jump_{\mathcal{N}}(J)\cup \Jump_{\mathcal{N}}(I_k)\\
      M_{i, i+d} = f(\gamma_{i, i+d})& \mbox{if }    (i, i+d) \in \Jump_{\mathcal{N}}(I_k)   \\
      N_{i, i+d} = -f(\gamma_{i, i+d})& \mbox{if }    (i, i+d) \in \Jump_{\mathcal{N}}(J)       
    \end{array}
\right.
$$
Let us define $A, B\in \GL_r(R)$ as follows
$$
 \left\{
    \begin{array}{ll}
        A_{i, i+d}=- \gamma_{i, i+d}  & \mbox{if }    (i, i+d) \in \Jump_{\mathcal{N}}(I_k)  \setminus ( \Jump_{\mathcal{N}}(J)\cap \Jump_{\mathcal{N}}(I_k))\\
        A_{i,j}=0  & \mbox{otherwise} 
    \end{array}
\right.
$$
and 
$$
 \left\{
    \begin{array}{ll}
        B_{i, i+d}= \gamma_{i, i+d}  & \mbox{if }    (i, i+d) \notin \Jump_{\mathcal{N}}(I_k)  \cup\Jump_{\mathcal{N}}(J)\\
        B_{i, i+d}=\gamma_{i, i+d} & \mbox{if }  (i,i+d)\in  \Jump_{\mathcal{N}}(J)  \setminus ( \Jump_{\mathcal{N}}(J)\cap \Jump_{\mathcal{N}}(I_k))\\
        B_{i,j}=0  & \mbox{otherwise} 
    \end{array}
\right.
$$
Note that $\Id+A\in \Gamma(J, \St_{\mathcal{N}}(R))$ and $\Id+B\in \Gamma(I_k, \St_{\mathcal{N}}(R))$. For $i=1, \dots, d+1$, we have $\Diag_i(f(A)M)=\Diag_i(f(B)N)=0$. If $i<d+1$, we deduce by recursion hypothesis
\begin{align*}
\Diag_i((\Id+f(A))g)  & =\Diag_i(\Id) + \Diag_i(M) \\
                             & =\Diag_i(\Id) + \Diag_i(N)\\
                              & = \Diag_i((\Id+f(B))h)  
\end{align*}
Finally, we have by definition of $A$ and $B$
\begin{align*}
\Diag_{d+1}((\Id+f(A))g)& =\Diag_{d+1}(f(A)+M)\\
                               &   =\Diag_{d+1}(f(B)+N)\\
                               &=\Diag_{d+1}((\Id+f(B))h)
\end{align*}
Replacing the $t_i^{\prime}$ by $(\Id+A) t_i^{\prime}$ for $i=1, \dots, k-1$ and $t_k^{\prime}$ by $(\Id+B)t_k^{\prime}$ changes $g$ into $(\Id+f(A))g$ and $h$ into $(\Id+f(B))h$. Keeping on with this process leads to a situation where $g=h$.
\end{proof}


We deduce the following
\begin{corollaire}\label{trivialiff}
If $R\longrightarrow S$ is an etale cover and if $I$ is an interval of $S^1$, then  $\mathcal{T}\in H^{1}(I, \St_{\mathcal{N}}(R))$ is trivial iff $\mathcal{T}(S)$ is trivial.
\end{corollaire}
\begin{proof}
The only statement requiring a proof is the converse statement. By Babbitt-Varadarajan representability and descent  \cite{SGA1}, the functor $H^{1}(S^{1}, \St_{\mathcal{N}})$  is a sheaf so \ref{trivialiff} is known for $I=S^{1}$. If $I$ is a strict interval, this is a consequence of \ref{deRdansS} since an etale cover is an injective morphism of rings.
\end{proof}

\subsection{Sheaf property}
The goal of this subsection is to prove the following

\begin{lemme}\label{sheaf}
The functor $\St_{\mathcal{N}}$-$\AdSk_{\mathscr{C}}$ is a sheaf for the etale topology on $\mathds{C}$-alg.
\end{lemme}
\begin{proof}
From \ref{Stskrepresentable} and descent \cite{SGA1}, the functor $\St_{\mathcal{N}}$-$\Sk_{\mathscr{C}}$ is a sheaf for the etale topology. In particular, the subfunctor  $\St_{\mathcal{N}}$-$\AdSk_{\mathscr{C}}$ is a presheaf. To prove that it is a sheaf, we are left to show that for every etale cover $R\longrightarrow S$, a Stokes skeleton $(\mathcal{T}_{\mathscr{C}}, \iota_{\mathscr{C}})\in \St_{\mathcal{N}}\text{-}\Sk_{\mathscr{C}}(R)$ is admissible if $(\mathcal{T}_{\mathscr{C}}(S), \iota_{\mathscr{C}}(S))$ is admissible.\\ \indent
Let $\mathcal{S}$ be a $\mathcal{I}$-good open set, let  
$\pi : P\longrightarrow \mathcal{S}$ be a $\mathscr{C}$-polygon  with vertices $E_i=[x_i, x_{i+1}]$,  $i\in \mathds{Z}/N\mathds{Z}$,  let $t_i\in \Gamma(\pi(E_i), \mathcal{T}_{C(E_i)})$, $i=1, \dots, N-1$ such that
$$
t_i(\pi(x_{i+1}))=t_{i+1}(\pi(x_{i+1})) \text{  in } E(\mathcal{T}_\mathscr{C}, \iota_{\mathscr{C}})
$$
Let $t^{\prime}\in  \Gamma([x_N,x_1], \mathcal{T}_{C(E_N)}(S))$ be the unique section satisfying 
$$
t^{\prime}(\pi(x_N))=t_{N-1}(S)(\pi(x_N)) \text{ and }  t^{\prime}(\pi(x_1))=t_{1}(S)(\pi(x_1))  \text{  in } E(\mathcal{T}_\mathscr{C}(S), \iota_{\mathscr{C}}(S))
$$
Since $\mathcal{T}_{C(E_N)}(S)$ is trivial on the strict closed interval $\pi([x_N,x_1])$, it is trivial on a strict open interval containing  $\pi([x_N,x_1])$. We deduce from \ref{trivialiff} that $\mathcal{T}_{C(E_N)}$ is trivial on $\pi([x_N,x_1])$. Let $t\in \Gamma(\pi([x_N,x_1]), \mathcal{T}_{C(E_N)})$ and let $g\in \Gamma(\pi([x_N,x_1]), \St_{\mathcal{N}}(S))$ such that $t^{\prime}=g t(S)$. From \eqref{uneegalite} applied to  $Z=\pi([x_N,x_1])$ and $Y=\pi(x_1)$, we get that $g$ is defined over $R$. The section $gt$ is the sought-after section.
\end{proof}
\subsection{Twisted Lie algebras, tangent space and obstruction theory}\label{twisted}
For $R\in \mathds{C}$-alg, we denote by $\Lie \St_{\mathcal{N}}(R)$ the sheaf of Lie algebras over $R$ on $\mathds{T}$ induced by $\St_{\mathcal{N}}(R)$. Concretely,  $\Lie \St_{\mathcal{N}}(R)$ is the subsheaf of   $R\otimes_{\mathds{C}} (j_{D\ast}\mathcal{H}^{0}\DR\End \mathcal{N})_{|\mathds{T}}$ of sections $f$ satisfying $p_a f i_b = 0$ unless  $a<_\mathcal{S} b$. \\ \indent
Let $\underline{\mathcal{S}}=(\mathcal{S}_i)_{i\in K}$ be a cover of 
$\mathds{T}$  by good open subsets. For $i_1, \dots, i_k\in K$, we set as usual $\mathcal{S}_{i_1 \dots i_k}:= \cap_{j} \mathcal{S}_{i_j}$. We define $L_i(R):=\Lie \St_{\mathcal{N}}(R)_{|\mathcal{S}_i}$. Let $\mathcal{T}\in H^{1}(\mathds{T}, \St_{\mathcal{N}}(R))$ and let $g=(g_{ij})$ be a cocycle representing $\mathcal{T}$. The identifications
\begin{eqnarray*}
L_i(R)_{|S_{ij}}    &     \overset{\sim}{\longrightarrow}  & L_j(R)_{|S_{ij}}  \\
   M   &            \longrightarrow  & g_{ij}^{-1}M g_{ij}
\end{eqnarray*}
allow to glue the $L_i(R)$ into a sheaf of $R$-Lie algebras over $\mathds{T}$ denoted by $\Lie \St_{\mathcal{N}}(R)^{\mathcal{T}}$ and depending only on $\mathcal{T}$ and not on $g$. 
Let us examine the first cocycle conditions in the Cech complex of $\Lie \St_{\mathcal{N}}(R)^{\mathcal{T}}$. For $t=(t_i)_{i\in K}\in \check{C}^{0}(\underline{\mathcal{S}}, \Lie \St_{\mathcal{N}}(R)^{\mathcal{T}})$, let $d_{i}$ be the unique representative of $t_i$ in $\Gamma(\mathcal{S}_{i}, L_i(R))$. Then
\begin{align*}
(dt)_{ij}& =t_i- t_j \\
           & = [d_i]-[d_j]\\
           & =[d_i- g_{ij}d_j g_{ij}^{-1}]
\end{align*}
where $[$ $]$ denotes the class of an element of $\Gamma(\mathcal{S}_{i}, L_i(R))$ in  $\Gamma(S_{ij},\Lie \St_{\mathcal{N}}(R)^{\mathcal{T}})$. 
For $t=(t_{ij})\in \check{C}^{1}(\underline{\mathcal{S}}, \Lie \St_{\mathcal{N}}(R)^{\mathcal{T}})$, let $\alpha_{ij}$ be the unique representative of $t_{ij}$ in $\Gamma(\mathcal{S}_{ij}, L_i(R))$. The relation $t_{ij}=- t_{ji}$ translates into $\alpha_{ji}=-g^{-1}_{ij}\alpha_{ij}g_{ij}$. Thus
\begin{align*}
(dt)_{ijk} &=t_{jk}+t_{ki}+t_{ij}\\  
&=[g_{ij}\alpha_{jk}g_{ij}^{-1}+g_{ik}\alpha_{ki}g_{ik}^{-1}+\alpha_{ij}]
\end{align*}
Let us now compute the tangent space of $H^{1}(\mathds{T},\St_{\mathcal{N}})$ at $\mathcal{T}$. 
\begin{lemme}\label{lemmecalculTangent}
There is a canonical isomorphism of $R$-modules
\begin{equation}\label{calculTangent}
\xymatrix{
T_{\mathcal{T}}H^{1}(\mathds{T}, \St_{\mathcal{N}})\ar[r]^-{\sim} & \check{H}^{1}(\underline{\mathcal{S}}, \Lie \St_{\mathcal{N}}(R)^{\mathcal{T}})
}
\end{equation}
\end{lemme}
\begin{proof}
Let $\mathcal{T}_{\epsilon} \in T_{\mathcal{T}}H^{1}(\mathds{T}, \St_{\mathcal{N}})$. A cocycle of $\mathcal{T}_{\epsilon}$ associated to the good cover $\underline{\mathcal{S}}$ can be written $(g_{ij}+\epsilon_{ij})$ where the $\epsilon_{ij}$ take value in the line $\mathds{C}\epsilon$. The relation $g_{ji}+\epsilon_{ji}=(g_{ij}+\epsilon_{ij})^{-1}$ is equivalent to 
$$
\epsilon_{ji}=-g_{ij}^{-1}\epsilon_{ij}g_{ij}^{-1}
$$
The cocycle condition is equivalent to 
$$
(g_{ij}+\epsilon_{ij})(g_{jk}+\epsilon_{jk})(g_{ki}+\epsilon_{ki})=\Id
$$
which is equivalent to the following equality in $\Gamma(S_{ij},  \Lie \St_{\mathcal{N}}(R))$
$$
\epsilon_{ij}g_{jk}g_{ki} + g_{ij}\epsilon_{jk}g_{ki} + g_{ij}g_{jk}\epsilon_{ki}=0
$$
That is 
\begin{equation}\label{relation1retour}
\epsilon_{ij}g_{ij}^{-1}+g_{ij}\epsilon_{jk}g_{ki} +g_{ki}^{-1}\epsilon_{ki}=0
\end{equation}
Let us set $\alpha_{ij}=\epsilon_{ij}g_{ij}^{-1}$ viewed as a section of $ L_i$ over $\mathcal{S}_{ij}$ and let $t_{ij}(\epsilon)$ be the class of $\alpha_{ij}$ in $\Lie \St_{\mathcal{N}}(R)^{\mathcal{T}}$. We have $t_{ii}(\epsilon)=[\alpha_{ii}]=0$ and
$$
t_{ji}(\epsilon)=[\alpha_{ji}]=[\epsilon_{ji}g_{ij}]=
-[g_{ij}^{-1}\epsilon_{ij}]=-[\alpha_{ij}]=-t_{ij}(\epsilon)
$$
The relations explicited in \ref{twisted} show that \eqref{relation1retour}  is equivalent to  
the  $t_{ij}(\epsilon)$ defining a cocycle of $\Lie \St_{\mathcal{N}}(R)^{\mathcal{T}}$. Its class in 
$\check{H}^{1}(\underline{\mathcal{S}}, \Lie \St_{\mathcal{N}}(R)^{\mathcal{T}})$ 
does not depend on the choice of a cocycle representing $\mathcal{T}_{\epsilon}$. Indeed, suppose that $(g_{ij}+\epsilon_{ij})_{ij}$ is cohomologous to $(g_{ij}+\nu_{ij})_{ij}$, that is
\begin{eqnarray*}
g_{ij}+\nu_{ij}=(c_i+d_i)(g_{ij}+\epsilon_{ij})(c_j+d_j)^{-1}
\end{eqnarray*}
where $(c_i+d_i)_{i\in K}\in \check{C}^{0}(\underline{\mathcal{S}}, \St_{\mathcal{N}}(R[\epsilon]))$. We obtain
\begin{equation}\label{unauto}
g_{ij}=c_i g_{ij} c_j^{-1}
\end{equation}
and
\begin{equation}\label{vasesimplifier}
\nu_{ij}=c_i \epsilon_{ij}c_j^{-1}+ d_i g_{ij} c_j^{-1}- c_i g_{ij} c_j^{-1} d_j c_{j}^{-1}
\end{equation}
Relation \eqref{unauto} implies that the $c_i$ define an automorphism of $\mathcal{T}$. From  \ref{autoId}, we deduce that $c_i= \Id$ for every $i$. Hence, \eqref{vasesimplifier} gives 
$$
\nu_{ij}= \epsilon_{ij}+ d_i g_{ij} - g_{ij}  d_j 
$$
Multiplying by $g_{ij}^{-1}$ and taking the class in $\Lie \St_{\mathcal{N}}(R)^{\mathcal{T}}$ gives 
$$
t_{ij}(\nu)= t_{ij}(\epsilon)+ [d_i - g_{ij}  d_j g_{ij}^{-1}]
$$
Hence, the morphism \eqref{calculTangent} is well-defined and is injective. One easily checks that it  is surjective.
\end{proof}
One can show similarly (but we will not need it) that $\check{H}^{2}(\underline{\mathcal{S}}, \Lie \St_{\mathcal{N}}(R)^{\mathcal{T}})$ provides an obstruction theory for the functor 
$H^{1}(\mathds{T},\St_{\mathcal{N}})$ at $\mathcal{T}$.

\section{Proof of Theorem \ref{bigtheorem}}\label{prooftheorem}

\subsection{Representability by an algebraic space}\label{repalgspace}
Let $\underline{\mathcal{S}}$ be a finite cover by $\mathcal{I}$-good open sets of $\mathds{T}$. Since the cocycle condition in 
$$
\prod_{\mathcal{S}, \mathcal{S}^{\prime}\in \underline{\mathcal{S}}}  \Gamma(\mathcal{S}\cap \mathcal{S}^{\prime}, \St_{\mathcal{N}})
$$
is algebraic, it defines a closed subscheme denoted by $Z( \underline{\mathcal{S}}, \St_{\mathcal{N}})$. From \ref{crittri}, the morphism of  presheaves 
$$
\xymatrix{
Z(\underline{\mathcal{S}}, \St_{\mathcal{N}}) \ar[r] & H^{1}(\mathds{T}, \St_{\mathcal{N}})
}
$$
is surjective. Hence $H^{1}(\mathds{T}, \St_{\mathcal{N}})$ is the quotient in the category of  presheaves  of $Z(\underline{\mathcal{S}}, \St_{\mathcal{N}})$ by the algebraic group
$$
\prod_{\mathcal{S}\in \underline{\mathcal{S}}}  \Gamma(\mathcal{S}, \St_{\mathcal{N}})
$$
From \ref{autoId}, this action is free. By Artin theorem \cite[6.3]{ArtinOnStack}, see also \cite[10.4]{LMB} and \cite[04S6]{SP}, we deduce that the sheaf associated to $H^{1}(\mathds{T}, \St_{\mathcal{N}})$ is representable by an algebraic space of finite type over $\mathds{C}$. From lemma  \ref{sheaf}, the functor $H^{1}(\mathds{T}, \St_{\mathcal{N}})$ is a sheaf for the etale topology on $\mathds{C}$-alg, so it is  isomorphic to its sheafification. Hence, $H^{1}(\mathds{T}, \St_{\mathcal{N}})$ is representable by an algebraic space of finite type over $\mathds{C}$. 
\subsection{A closed immersion in an affine scheme of finite type}\label{closedimmersionpreuve}
The morphism of algebraic spaces \eqref{injfamillespeciale} is a monomorphism of finite type. Hence, it is separated and quasi-finite \cite[Tag 0463]{SP}. Since a separated quasi-finite morphism is representable \cite[Tag 03XX]{SP}, we deduce that \eqref{injfamillespeciale}  is representable. Since $\St_{\mathcal{N}}\text{-}\Sk_{\mathscr{C}(\mathcal{I})}$ is a scheme, we deduce that $H^{1}(\mathds{T}, \St_{\mathcal{N}})$ is representable  by a scheme of finite type over $\mathds{C}$. We still denote  by $H^{1}(\mathds{T}, \St_{\mathcal{N}})$ this scheme. We are left to show that \eqref{injfamillespeciale} is a closed immersion. From \cite[18.12.6]{EGA4-4}, closed immersions are the same thing as proper 
monomorphisms. We are thus left to prove that \eqref{injfamillespeciale}  is proper.
By the valuative criterion for properness \cite[7.3.8]{EGA2}, we have to prove that if 
$R$ is a discrete valuation ring with field of fraction $K$, for 
$(\mathcal{T}_{\mathscr{C}(\mathcal{I})}, \iota_{\mathscr{C}(\mathcal{I})})\in 
\St_{\mathcal{N}}\text{-}\Sk_{\mathscr{C}(\mathcal{I})}(R)$ such that 
$$
(\mathcal{T}_{\mathscr{C}(\mathcal{I})}(K), \iota_{\mathscr{C}(\mathcal{I})}(K))=\sk_{\mathscr{C}(\mathcal{I})}(\mathcal{T}_K)
$$
where $\mathcal{T}_K\in H^{1}(\mathds{T}, 
\St_{\mathcal{N}}(K))$, there exists a (necessarily unique) $\mathcal{T}\in H^{1}
(\mathds{T}, \St_{\mathcal{N}}(R))$ such that $\sk_{\mathscr{C}(\mathcal{I})}
(\mathcal{T})=(\mathcal{T}_{\mathscr{C}(\mathcal{I})}, \iota_{\mathscr{C}(\mathcal{I})})$ (which then 
automatically implies $\mathcal{T}(K)=\mathcal{T}_K$). Since $R\longrightarrow K$ is injective, it is enough to prove that $\mathcal{T}_K$ is defined over $R$. \\ \indent
We make here an essential use of condition $(3)$ of \ref{mono}, so we depict it.
\begin{figure}[h]
\begin{center}
\includegraphics[height=1.5in,width=2in,angle=0]{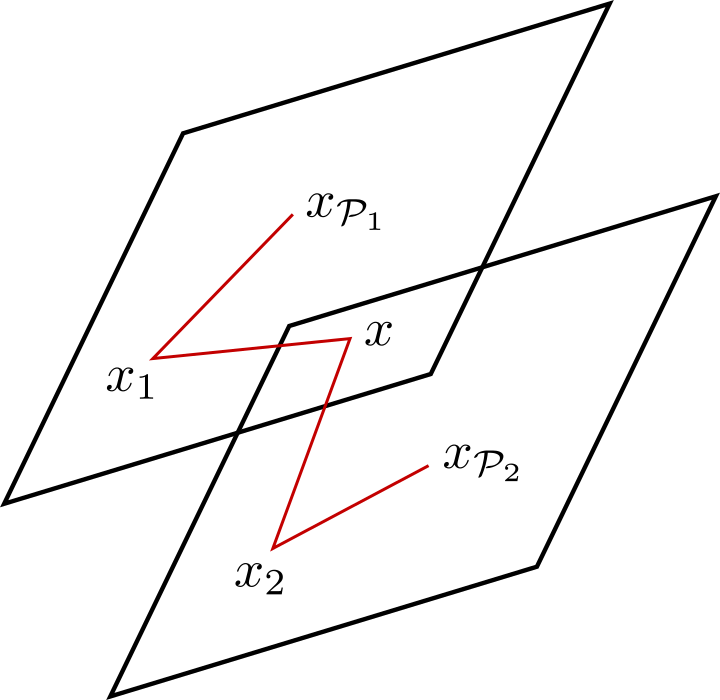}
\end{center}
\end{figure}
For $\mathcal{P}\in \mathscr{P}$, let us choose $t_{\mathcal{P}}\in \Gamma(\mathcal{P}, \mathcal{T}_K)$. By assumption, $x_\mathcal{P}$ belongs to a curve $C_{\mathcal{P}}\in \mathscr{C}(\mathcal{I})$. Let $s_{\mathcal{P}}\in \mathcal{T}_{C_{\mathcal{P}}, x_\mathcal{P}}$, and let us write 
$$
t_{\mathcal{P}, x_\mathcal{P}}=g_\mathcal{P} s_{\mathcal{P}}(K)  \text{ with  } g_\mathcal{P} \in \St_{\mathcal{N}}(K)_{x_\mathcal{P}}
$$
Since $\mathcal{P}$ is contained in an $\mathcal{I}$-good open set for $x_\mathcal{P}$, the section $g_\mathcal{P}$ extends to $\mathcal{P}$. Replacing 
$t_{\mathcal{P}}$ by $g_\mathcal{P}^{-1}t_{\mathcal{P}}$, we can thus suppose that
 $t_{\mathcal{P}, x_\mathcal{P}}=s_{\mathcal{P}}(K)$. 
Let $\mathcal{P}_1, \mathcal{P}_2\in \mathscr{P}$ such that $\mathcal{P}_1 \cap \mathcal{P}_2\neq \emptyset $. We have to show that the transition matrix between $t_{\mathcal{P}_1}$ and $t_{\mathcal{P}_2}$ take value in $R$.
For $i=1, 2$ we choose  $C_i\in \mathscr{C}(\mathcal{I})$ containing $[x_{\mathcal{P}_i}, x_i]$ and  $C_i^{\prime}\in \mathscr{C}(\mathcal{I})$ containing $[x_i, x_{12}]$. The torsors $\mathcal{T}_{C_i}(K)$ are trivial on $[x_{\mathcal{P}_i}, x_i]$. Since $R\longrightarrow K$ is injective, \ref{deRdansS}
ensures that the same is true for $\mathcal{T}_{C_i}$. Similarly, the torsors $\mathcal{T}_{C_i^{\prime}}$ are trivial on $[x_i, x_{12}]$. Let us choose $s_i \in  \Gamma([x_{\mathcal{P}_i}, x_i], \mathcal{T}_{C_i})$ and $s_i^{\prime} \in \Gamma([x_i, x],\mathcal{T}_{C_i^{\prime}})$. We have 
$$
t_{\mathcal{P}_i}= g_i s_i(K) \text{ in } \Gamma([x_{\mathcal{P}_i}, x_i], \mathcal{T}_K)
$$ 
where $g_i \in \Gamma([x_{\mathcal{P}_i}, x_i], \St_{\mathcal{N}}(K))$. Then
$$
t_{\mathcal{P}_i, x_{\mathcal{P}_i}}= g_{i, x_{\mathcal{P}_i}} s_{i, x_{\mathcal{P}_i}}(K)=s_{\mathcal{P}_i}(K)
$$
Since $s_{\mathcal{P}_i}$ and $s_i$ are both defined over $R$, so is $g_{i, x_{\mathcal{P}_i}}$. From \eqref{uneegalite}, we deduce that $g_i$ has coefficients in $R$. Hence, we have similarly on $[x_i, x_{12}]$
$$
t_{\mathcal{P}_i}=g_i^{\prime}s_i^{\prime}(K) \text{ with } g_i^{\prime}\in \Gamma([ x_i, x_{12}], \St_{\mathcal{N}}(R))
$$
Finally, $s_1^{\prime}$ and $s_2^{\prime}$ compare at $x_{12}$ in terms of a matrix $h$ with coefficient in $R$. If we write $t_{\mathcal{P}_2}= g_{12}t_{\mathcal{P}_1}$ on $\mathcal{P}_1 \cap \mathcal{P}_2$, we deduce 
$g_{12}= g_2^{\prime}hg_1^{\prime-1}$, so $g_{12}$ has  coefficients in $R$. This concludes the proof of Theorem \ref{bigtheorem}.

\section{The case where $\mathcal{I}$ is very good}

\subsection{Differential interpretration of $\St_{\mathcal{N}}$}\label{diffint}
From now on, we suppose that $\mathcal{I}$ is very good. This means that for every $a,b\in \mathcal{I}$ with $a \neq b$, the pole locus of $ a-b$ is exactly $D$. In particular, for every $R\in \mathds{C}$-alg, 
$$ \St_{\mathcal{N}}(R)=\Id+ R\otimes_{\mathds{C}}\mathcal{H}^{0}\DR^{<D}\End \mathcal{N}
$$

\subsection{Tangent space and irregularity}\label{tanandirr}
Let $(\mathcal{M}, \nabla, \iso)$ be a good $\mathcal{N}$-marked connection. A choice of local trivialisations for $\mathcal{T}:=\Isom_{\iso}(\mathcal{M}, \mathcal{N})$ gives rise to an isomorphism of sheaves on $\mathds{T}$
$$
\xymatrix{
\mathcal{H}^0 \DR^{<D}\End \mathcal{M}\ar[r]^-{\sim} & \Lie \St_{\mathcal{N}}(\mathds{C})^{\mathcal{T}}
}
$$ 
Hence, for every $i \in \mathds{N}$, there are canonical identifications
\begin{align*}
H^{i}(\mathds{T},  \Lie \St_{\mathcal{N}}(\mathds{C})^{\mathcal{T}})&\simeq  H^i( \mathds{T}, \mathcal{H}^0 \DR^{<D}\End \mathcal{M})  \\
 & \simeq 
      H^i(\mathds{T}, \DR^{<D}\End \mathcal{M})  \\
      & \simeq (\mathcal{H}^i\Irr^{\ast}_D\End \mathcal{M})_{0} 
\end{align*}
The second identification comes from the fact \cite[Prop. 1]{HienInv} that $\DR^{<D}\End \mathcal{M}$ is concentrated in degree $0$. The third identification comes from \cite[2.2]{SabRemar}. We deduce the following 
\begin{lemme}\label{tspace}
The tangent space of $H^{1}(\mathds{T}, \St_{\mathcal{N}})$  at $\mathcal{T}:=\Isom_{\iso}(\mathcal{M}, \mathcal{N})$ identifies in a canonical way to $(\mathcal{H}^1\Irr^{\ast}_D\End \mathcal{M})_{0}$. 
\end{lemme}
\begin{proof}
Since $\mathds{T}$ is paracompact, sheaf cohomology is computed by Cech cohomology. Moreover, $\mathcal{I}$-good opens form a basis of the topology of $\mathds{T}$.
For two such covers $\underline{\mathcal{U}}$ and $\underline{\mathcal{V}}$ with $\underline{\mathcal{V}}$ refining  $\underline{\mathcal{U}}$, lemma \ref{lemmecalculTangent} provides a commutative diagram
$$
\xymatrix{
& T_{\mathcal{T}}H^{1}(\mathds{T}, \St_{\mathcal{N}})\ar[rd] \ar[ld] & \\
 \check{H}^{1}(\underline{\mathcal{U}}, \Lie \St_{\mathcal{N}}(\mathds{C})^{\mathcal{T}})  \ar[rr]&  &  \check{H}^{1}(\underline{\mathcal{V}}, \Lie \St_{\mathcal{N}}(\mathds{C})^{\mathcal{T}})
}
$$
where the diagonal arrows are isomorphisms. Hence, the vertical arrow is an isomorphism and \ref{tspace} is proved.

\end{proof}
\subsection{Proof of Theorem \ref{restriction}}
Unramified morphisms of finite type are quasi-finite. Hence, it is enough to prove that a good $\mathcal{N}$-marked connection $(\mathcal{M}, \nabla, \iso)$ belongs to the unramified locus of $\res_V$, which is open. The cotangent sequence of $\res_V$ reads
$$
\xymatrix{
 \res_V^{\ast} \Omega^{1}_{H^{1}(\mathds{T}^{\prime}, \St_{\mathcal{N}_V})}     \ar[r] &   \Omega^{1}_{H^{1}(\mathds{T}, \St_{\mathcal{N}})}   \ar[r] &     \Omega^{1}_{\res_V} \ar[r] & 0
}
$$
Taking the fiber at  $\mathcal{T}:=\Isom_{\iso}(\mathcal{M}, \mathcal{N})$ preserves cokernel, so after dualizing,  we obtain the following exact sequence
\begin{equation}\label{ptitexasequ}
\xymatrix{
  0    \ar[r] & \Omega^{1}_{\res_V}(\mathcal{T})^{\vee} \ar[r] &   T_{\mathcal{T}}H^{1}(\mathds{T}, \St_{\mathcal{N}})  \ar[r] &   T_{ \res_V(\mathcal{T})}H^{1}(\mathds{T}^{\prime}, \St_{\mathcal{N}_V})
}   
\end{equation}
By Nakayama lemma, we have to prove that $\Omega^{1}_{\res_V}(\mathcal{T} )$ vanishes. Let $i_V: V\longrightarrow \mathds{C}^n$ be the canonical inclusion.
 Since  $\mathcal{M}$ is localized at $0$, we have
$$
(\Irr^{\ast}_{D}\End \mathcal{M})_0 \simeq \Irr^{\ast}_{0}\End \mathcal{M}
$$
Applying $\Irr^{\ast}_{0}$ \cite[3.4-2]{Mehbsmf}  to the local cohomology triangle \cite{Kalivre}
$$
\xymatrix{
i_{V+}i_{V}^+\End \mathcal{M}[\dim V-n]\ar[r]& \End \mathcal{M} \ar[r]& \End \mathcal{M}(\ast V)\ar[r]^-{+1}& 
}
$$
gives a distinguished triangle
$$
\xymatrix{
(\Irr^{\ast}_{V}\End \mathcal{M})_0 \ar[r]&  \Irr^{\ast}_{0}\End \mathcal{M}\ar[r]& 
\Irr^{\ast}_{0}\End \mathcal{M}_{V}
\ar[r]^-{+1}& 
}
$$
Hence, we obtain an exact sequence 
\begin{equation}\label{ptitexasequ2}
\xymatrix{
0\ar[r]& (\mathcal{H}^1\Irr^{\ast}_{V}\End \mathcal{M})_0
 \ar[r]& \mathcal{H}^1\Irr^{\ast}_{0}\End \mathcal{M}\ar[r]& \mathcal{H}^1\Irr^{\ast}_{0}\End \mathcal{M}_{V} 
}
\end{equation}
From \ref{tspace}, the sequence \eqref{ptitexasequ} identifies canonically to \eqref{ptitexasequ2}. Hence, we are left to show the following
\begin{proposition} 
For every $\mathcal{N}$-marked connection $\mathcal{M}$, we have $$(\mathcal{H}^1\Irr^{\ast}_{V}\mathcal{M})_0\simeq 0$$
\end{proposition}
\begin{proof}
Let $p: X\longrightarrow \mathds{C}^{n}$ be the blow-up at the origin. Let us denote by $E$ the exceptional divisor,  $V^{\prime}$ (resp. $D^{\prime}$) the strict transform of $V$ (resp. $D$). Since $p$ is an isomorphism above $\mathds{C}^{n}\setminus D$, we know from  \cite[3.6-4]{Mehbsmf}   that $p_+ p^{+ }\mathcal{M} \longrightarrow \mathcal{M}$
is an isomorphism. Applying $\Irr^{\ast}_{V}$ and the compatibility of $\Irr^{\ast}$ with proper push-forward, we obtain  
\begin{align*}
(\mathcal{H}^1\Irr^{\ast}_{V}\mathcal{M})_0 & \simeq H^{1}({E},(\Irr^{\ast}_{p^{-1}(V)}p^{+}\mathcal{M})_{|E})\\
              &  \simeq     H^{0}({E},(\mathcal{H}^{1}\Irr^{\ast}_{E\cup V^{\prime}}p^{+}\mathcal{M})_{|E})
\end{align*}
Set $\Omega:= E\setminus (E\cap D^{\prime})$. Since $V$ is transverse to every irreducible components of $D$, the set $\Omega \cap V^{\prime}$ is not empty. Pick $d\in \Omega \cap V^{\prime}$. Since $E$ is included in the pole locus of $p^{+}\mathcal{M}$,  
\begin{equation}\label{isoirr}
(\Irr^{\ast}_{E\cup V^{\prime}}p^{+}\mathcal{M})_d \simeq (\Irr^{\ast}_{V^{\prime}}p^{+}\mathcal{M})_d
\end{equation}
Moreover, the multiplicities of the components of $\Char p^{+}\mathcal{M}$
passing through $x\in \Omega$ only depend on the formalization of $p^{+}\mathcal{M}$ at $x$. Since $p^{+}\mathcal{M}$ has good formal decomposition along the smooth divisor $\Omega$, we deduce that above $\Omega$, $\Char \mathcal{M}$ is supported on the conormal bundle of  $\Omega$. In particular, $V^{\prime}$ is non-characteristic for $p^{+}\mathcal{M}$ at $d$. Thus,  \cite{TheseKashiwara} asserts that 
the Cauchy-Kowaleska  morphism
$$
\RHom(p^+\mathcal{M}, \mathcal{O}_X)_{|V^{\prime}}\longrightarrow \RHom((p^+\mathcal{M})_{V^{\prime}}, \mathcal{O}_{V^{\prime}})
$$
is an isomorphism in a neighbourhood of $d$. From \cite[V 2.2]{MT}, we deduce that the right-hand side of  \eqref{isoirr} is zero. \\ \indent
Take  $s\in H^{0}({E},(\mathcal{H}^{1}\Irr^{\ast}_{E\cup V^{\prime}}p^{+}\mathcal{M})_{|E})$. Since 
$$
(\Irr^{\ast}_{E\cup V^{\prime}}p^{+}\mathcal{M})_{|E}\simeq \Irr^{\ast}_{E}(p^{+}\mathcal{M})(\ast V^{\prime})
$$
we know from \cite{Mehbgro} that the complex $(\Irr^{\ast}_{E\cup V^{\prime}}p^{+}\mathcal{M})_{|E}[1]$ is a perverse sheaf on $E$. So to prove $s=0$, we are left to prove that the support of $s$ is contained in a closed subset of dimension $<\dim E$ \cite[(10.3.3)]{KS}. Hence, it is enough to prove that $s$ vanishes on $\Omega$. From the discussion above, $s$ vanishes on a neighbourhood $U$ of $\Omega \cap V^{\prime}$ in  $\Omega$. Since $\Omega$ is path-connected, a point in $\Omega \setminus (\Omega \cap V^{\prime})$ can be connected  to a point in $U \setminus (U \cap V^{\prime})$ via a path in $\Omega \setminus (\Omega \cap V^{\prime})$. So we are left to see that  $\mathcal{H}^{1}\Irr^{\ast}_{E\cup V^{\prime}}p^{+}\mathcal{M}$ is a local system on $\Omega \setminus (\Omega \cap V^{\prime})$. This is a consequence of the following
\begin{lemme}
Let $X$ be a smooth manifold and let $Z$ be a smooth divisor of $X$. Let $\mathcal{M}$ be a meromorphic connection on $X$ with poles along $Z$ admitting a good formal structure along $Z$. Then $\Irr^{\ast}_Z\mathcal{M}$ is a local system concentrated in degree 1.
\end{lemme}
We can see this as a particular case of Sabbah theorem \cite{SabRemar}. Let us give an elementary argument. Since $\mathcal{M}$ has good formal structure along $Z$, any smooth curve transverse to $Z$ is non-characteristic for $\mathcal{M}$. Hence, $\Irr^{\ast}_Z\mathcal{M}$  is concentrated in degree 1 and $x\rightarrow \dim \mathcal{H}^1(\Irr^{\ast}_Z\mathcal{M})_x$ is constant to the generic irregularity of $\mathcal{M}$ along $Z$. We know from  \cite{Mehbgro} that $\Irr^{\ast}_Z\mathcal{M}[1]$ is perverse on $Z$. We conclude using the fact that a perverse sheaf on $Z$ with constant Euler-Poincaré characteristic function is a local system concentrated in degree 0.

\end{proof}

\subsection{Proof of Theorem \ref{rigidity}}
Let $(\mathcal{M}, \nabla, \iso)$ be a good $\mathcal{N}$-marked connection. Since $H^{1}(\mathds{T}, \St_{\mathcal{N}})$ is of finite type, it is enough to prove that the tangent space of $H^{1}(\mathds{T}, \St_{\mathcal{N}})$ at $\Isom_{\iso}(\mathcal{M}, \mathcal{N})$ vanishes. From \ref{tanandirr}
and Sabbah invariance theorem \cite{SabRemar}, we are left to prove the vanishing of $(\mathcal{H}^1\Irr^{\ast}_D\End \mathcal{N})_{0}$ for a generic choice of the regular parts $\mathcal{R}_a$, $a\in \mathcal{I}$. Generically, the monodromy is semi-simple, so  we are left to prove that for every $i \in \mathds{N}$,
\begin{equation}\label{vanishing}
(\mathcal{H}^i\Irr^{\ast}_D z^{\alpha} \mathcal{E}^{a})_0\simeq 0
\end{equation}
where $\alpha=(\alpha_1, \dots, \alpha_m)\in \mathds{C}^m$ is generic,  $z^{\alpha}=z_1^{\alpha_1}\cdots z_m^{\alpha_m}$ and $a$ is a good meromorphic function with poles along $D$. 
By a change of variable, we can suppose $a=1/z_1^{a_1}\cdots  z_m^{a_m}$ where  $a_i\in \mathds{N}^{\ast}$, $i=1, \dots, m$.  We are thus left to prove the following 
\begin{lemme}\label{quasidernierlemme}
Suppose that there exists $i, j\in \llbracket 1,m\rrbracket$ with  $i\neq j$ such that  
$$
\alpha_i a_i +\alpha_j a_j \notin \mathds{Z}
$$
Then, $(\Irr^{\ast}_D z^{\alpha}\mathcal{E}^{a})_{0}\simeq 0$.
\end{lemme}
In a first draft of this paper, a proof using  perversity arguments was given. We give here a simpler and more natural argument, kindly communicated to us by C. Sabbah.
\begin{proof}
From \cite[2.2]{SabRemar} and \cite[Prop. 1]{HienInv}, we have
$$
(\Irr^{\ast}_D z^{\alpha}\mathcal{E}^{a})_{0}\simeq R\Gamma( \mathds{T}, \mathcal{H}^{0}\DR^{<D}z^{\alpha}\mathcal{E}^{a})
$$
Let $\rho : \mathds{T}\longrightarrow S^1$ be the morphism $(\theta_1, \dots, \theta_m)\longrightarrow \sum_{i=1}^m a_i \theta_i$.  Then, $\mathcal{H}^{0}\DR^{<D}z^{\alpha}\mathcal{E}^{a}$ is the extension by $0$ of the restriction $L_{\alpha}$ of $\DR z^{\alpha}$ to $U:=\rho^{-1}(]\pi/2, 3\pi/2[)$. By Leray spectral sequence 
$$
E_{2}^{pq}=H_c^p(]\pi/2, 3\pi/2[, R^q\rho_{\ast} L_{\alpha}) \Longrightarrow R\Gamma_c(U, L_{\alpha})
$$
and proper base change, we are left to prove
$$
R\Gamma(\rho^{-1}(\theta), L_{\alpha})\simeq 0
$$
for every $\theta\in S^{1}$. Note that $\rho^{-1}(\theta)\simeq (S^1)^{m-1}$. By homotopy, it is enough to treat the case $\theta=0$. Without loss of generality, we can suppose $\alpha_1 a_1 +\alpha_2 a_2 \notin \mathds{Z}$. Let $p: \rho^{-1}(0)\longrightarrow  (S^1)^{m-2}$ be the restriction to $\rho^{-1}(0)$ of the projection on coordinates $\theta_3, \dots, \theta_n$. By the same argument as above, we are left to show that the restriction of  $ L_{\alpha}$ to the circles $p^{-1}(\theta), \theta \in (S^1)^{m-2}$ has no cohomology. Since the monodromy of such a restriction is the multiplication by $e^{2\pi i (\alpha_1 a_1 +\alpha_2 a_2)}$, we are done.
\end{proof}

\bibliographystyle{amsalpha}
\bibliography{ModuliBibli}

\end{document}